\documentclass[10pt]{amsart}

\usepackage{amsfonts,amssymb}
\usepackage[all]{xy}

\newtheorem{thm}{Theorem}[section]
\newtheorem{prop}[thm]{Proposition}
\newtheorem{cor}[thm]{Corollary}
\newtheorem{lem}[thm]{Lemma}

\theoremstyle{remark}
\newtheorem{remark}{Remark}[section]

\theoremstyle{definition}

\renewcommand{\arraystretch}{1.6}

\newcommand*\isom{%
  \xrightarrow{\sim}%
}

\def\qq{\mathbb{Q}}

\def\rr{\mathbb{R}}
\def\zz{\mathbb{Z}}
\def\cc{\mathbb{C}}

\def\mm{\mathcal{M}}
\def\aa{\mathcal{A}}
\def\bb{\mathcal{B}}

\def\jj{\mathcal{J}}
\def\tt{\mathcal{T}}

\def\oo{\mathcal{O}}

\def\dd{\mathcal{D}}
\def\ee{\mathcal{E}}
\def\uu{\mathbb{U}}
\def\UU{\mathcal{U}}
\def\hh{\mathcal{H}}

\def\deldelbar{\partial \bar{\partial}}
\def\Im{\mathrm{Im}\,}

\newcommand{\beq}{\begin{equation}}
\newcommand{\eeq}{\end{equation}}
\newcommand{\pair}[1]{\langle#1\rangle}
\newcommand{\abs}[1]{|#1|}
\newcommand{\divisor}{\operatorname{div}}
\newcommand{\ord}{\operatorname{ord}}
\newcommand{\Pic}{\operatorname{Pic}}
\newcommand{\codim}{\operatorname{codim}}

\newcommand{\R}{\operatorname{R}}

\numberwithin{equation}{section}

\begin{document}

\title[Asymptotics of the N\'eron height pairing]{Asymptotics of the N\'eron height pairing}

\author{David Holmes and Robin de Jong}



\begin{abstract}  The aim of this paper is twofold. First, we study the asymptotics of the N\'eron height pairing between degree-zero divisors on a family of degenerating compact Riemann surfaces parametrized by an algebraic curve. We show that if the monodromy is unipotent the leading term of the asymptotic formula is controlled by the local non-archimedean N\'eron height pairing on the generic fiber of the family. Second, we prove a conjecture of R. Hain to the effect that the `height jumping divisor' related to the normal function $(2g-2)x-K$ on the moduli space $\mm_{g,1}$ of $1$-pointed curves of genus $g \geq 2$ is effective. Both results follow from a study of the degeneration of the canonical metric on the Poincar\'e bundle on a family of principally polarized abelian varieties.
\end{abstract}

\maketitle

\thispagestyle{empty}

\section{Introduction}

The N\'eron height pairing is a canonical real-valued pairing between divisors $D,E$ of degree zero and with disjoint support on a compact Riemann surface. The pairing can be defined by taking a real Green's function $g_D$ for $D$, harmonic outside the support of $D$, and evaluating $g_D$ on $E$. The N\'eron pairing $g_D[E]$ is symmetric and bi-additive, and plays a prominent role in the arithmetic geometry of curves over number fields where it serves as an archimedean contribution to the global canonical height pairing between points on the jacobian \cite{gr}. It is a special case of Arakelov's pairing \cite{ar} between divisors of arbitrary degree on a compact Riemann surface.

In this paper we study the behavior of the N\'eron height pairing between generically disjoint families of degree zero divisors on a family $X\to S$ of compact Riemann surfaces, where the base manifold $S$ is a complex algebraic curve. Write $p$ for a point on $S$, and let $\overline{S}$ be a smooth compactification of $S$. When the point $p$ tends to a boundary point $s$ of $S$ in $\overline{S}$, so that the fibers of the family $X \to S$ may degenerate, the function $p \mapsto g_{D_p}[E_p]$ on $S$ typically acquires a singularity. One would like to investigate the shape of this singularity as $p$ tends to $s$.

Let $t$ be a local parameter around $s$ on $\overline{S}$, and assume that the local monodromy of the family $X \to S$ around the point $s$ is unipotent.  Then general Hodge theoretic results of R. Hain \cite{hainbiext} and D. Lear \cite{lear} from the early 1990s and of G. Pearlstein \cite{pearl} from 2005 imply that the singularity of the function $g_{D_p}[E_p]$ as $p \to s$ is of a very simple type, namely $g_{D_p}[E_p] \sim e \log|t(p)|$ for some rational number $e$. Here the notation $\sim$ means that the difference between both terms extends as a bounded continuous function over a small open disc in $\overline{S}$ containing $s$. The number $e$ can be expressed in terms of the local monodromy and the weight filtration on a canonical variation of mixed Hodge structure $\mathcal{V}_p$ on $S$ associated to the pair of divisors $D,E$ (cf. \cite[Theorem 5.19]{pearl}). \\

A first aim of this paper is to make these general results somewhat more explicit in the special case of families of curves. Note that the results of Hain, Lear and Pearlstein referred to above are set up in a much broader context, where the proper flat family $X \to S$ possibly has higher dimensional fibers, and the object of study is the asymptotic behavior of the archimedean height pairing \cite{beil} \cite{bost} \cite{gs} \cite{hainbiext} associated to two flat families of cycles over $S$, homologically trivial in each fiber. Restricting to families of curves gives us access to more specialized tools, such as Deligne pairing and determinant of cohomology \cite{de}, the existence of a flat family of positive definite hermitian forms on the associated family of jacobian varieties, and the algebraic theory of N\'eron models \cite{blr}.

Still assuming that the local monodromy of the family $X \to S$ around the point $s$ is unipotent, Theorem \ref{main} of the present paper interprets the number $e$ in terms of the intersection behavior of (extensions of) the divisors $D,E$ on a smooth compactification $\overline{X}$ of the complex surface $X$. More precisely, the coefficient $e$ turns out to be equal to a variant $\pair{D,E}_{	\textrm{a},s}$ of the local intersection multiplicity of $D,E$ at $s$, taking into account the fact that although $D,E$ are generically of degree zero, they may fail to be of degree zero on each component of the special fiber of $\overline{X}$ above~$s$.

More precisely, the rational number $\pair{D,E}_{\textrm{a},s}$ equals the local \emph{non-archimedean} N\'eron pairing, with respect to the discrete valuation on the function field of $S$ associated to $s$, of the restrictions of $D,E$ to the generic fiber of $X \to S$, as defined in \cite{gr} and \cite{zh}. Our asymptotic formula thus shows a remarkable compatibility between the archimedean and non-archimedean N\'eron height pairings. In the special case where the special fiber of $\overline{X}$ at $s$ has just \emph{one} node, our formula is implied by \cite[Theorems 6.10 and 7.2]{we}, where asymptotics are derived for Arakelov's pairing (on divisors which are not necessarily of degree zero). \\

Let $Y$ be a complex manifold and consider a polarized variation of Hodge structure $\uu$ of weight $-1$ over $Y$. Let $\jj(\uu) \to Y$ be the associated family of intermediate jacobians, and assume that a normal function section $\nu \colon Y \to \jj(\uu)$ of $\jj(\uu) \to Y$ is given. The second part of this paper deals with a canonical extension of a certain metrized line bundle on $Y$ associated to these data over a compactification $\overline{Y}$ of $Y$. As it turns out \cite{hainnormal}, the formation of this so-called `Lear extension' (for which the basic source is Lear's PhD thesis \cite{lear}) is \emph{not} usually compatible with pullback along morphisms $T \to Y$. When $T$ is a curve, with smooth compactification $\overline{T}$, the `difference' that occurs can be viewed as a divisor supported on the boundary $\overline{T} \setminus T$ of $T$, called the \emph{height jumping divisor}.

The phenomenon of height jumping is analyzed in detail in \cite{brospearl}, \cite{hainnormal} and \cite{pearl}. In \cite{hainnormal} Hain conjectures that for a collection of examples where $Y$ is a moduli orbifold of smooth pointed curves, the height jumping divisor with respect to morphisms $T \to Y$ where $T$ is a curve should always be \emph{effective}. As is explained in \cite{hainnormal}, the effectivity of the height jumping divisor in these examples has interesting ramifications for finding refined slope inequalities on moduli spaces of curves.

Theorem \ref{hainconj} below implies that the height jumping divisor is indeed effective for the case (mentioned by Hain) where $Y$ is the moduli orbifold $\mm_{g,1}$ of $1$-pointed curves of genus $g \geq 2$, the variation of Hodge structure $\uu$ is the tautological one, and the normal function $\nu$ is the section of the universal jacobian $\jj(\uu)$ over $\mm_{g,1}$ given by sending a point $[C,x]$ of $\mm_{g,1}$ to the class of the degree-zero divisor $(2g-2)x-K_C$. Here $K_C$ is a canonical divisor on $C$.

For the proof we will use our asymptotic analysis of the N\'eron height pairing, or rather its more intrinsic version Theorem \ref{cor}, which describes how the Deligne pairing $\pair{D,E}$ of two degree-zero divisors $D,E$ on $S$ can be extended (up to taking a tensor power) over the smooth compactification $\overline{S}$ as a continuously metrized line bundle. Another significant part of the proof deals with an analysis of the Green's functions induced by $D,E$ on the semistable reduction graphs associated to the degenerate fibers of $\overline{X}$ over $\overline{S}$. The non-archimedean pairing $\pair{D,E}_{\textrm{a},s}$ can be conveniently expressed in terms of these Green's functions.

\section{Statement of the main results} \label{section:results}

We now turn to a more precise formulation of our main results.
Let $S$ denote a smooth connected curve over the complex numbers $\cc$ and let $\pi \colon X \to S$ be a smooth proper morphism with connected one-dimensional fibers. Let $D,E$ be divisors on $X$ of relative degree zero. For general $p \in S$  the restrictions $D_p, E_p$ of $D,E$ to the fiber $X_p$ of $\pi$ above $p$ are divisors on $X_p$, and for such $p$ we consider a real Green's current $g_{D_p}$ on the compact Riemann surface $X_p$ associated to the divisor $D_p$. This $g_{D_p}$ is a generalized function on $X_p$, solving the equation
\begin{equation} \label{defproperty} \deldelbar g_{D_p} + \pi i \delta_{D_p} = 0 \, .
\end{equation}
Note that this equation determines each $g_{D_p}$ up to an additive real constant. For our purposes the choice of this constant will not matter.

Assume that $E_p$ has support disjoint from the support of $D_p$. We will write $g_{D_p}[E_p]$ as a shorthand for $\sum_i a_i g_{D_p}(q_i)$, where $E_p = \sum_i a_i q_i$. The real number $g_{D_p}[E_p]$ is the N\'eron height pairing between $D_p$ and $E_p$. We will be interested in the asymptotic behavior of the function $g_{D_p}[E_p]$ as $p$ approaches the boundary of $S$ in the complex topology. Let $\overline{S}$ denote a smooth compactification of $S$, and let $\overline{X} \to \overline{S}$ be a proper flat morphism extending $X \to S$. We will always make the following assumptions: (a) the surface $\overline{X}$ is smooth over $\cc$, and (b) each fiber of $\overline{X} \to \overline{S}$ is reduced, and has only ordinary double points as singularities. Assumption (b) may equivalently be phrased as saying that the monodromy around each of the points in the boundary of $S$ in $\overline{S}$ is unipotent.

From now on, let $D,E$ be two divisors of relative degree zero on $\overline{X}$. Let $s$ be a closed point in $\overline{S}$. By general properties of the intersection pairing on the special fibers of $\overline{X} \to \overline{S}$, see e.g. \cite[Theorem 9.1.23]{liu}, there exists a unique - up to adding $\qq$-multiples of fibers - vertical $\qq$-divisor $\phi(D)$ on $\overline{X}$ such that $D +\phi(D)$ has zero intersection product with all vertical divisors of $\overline{X}$. Let $\pair{,}_s$ denote the local intersection pairing on $\overline{X}$ over $s$. The local N\'eron height pairing of $D,E$ relative to $s$ is then defined to be the rational number
\[ \pair{D,E}_{\textrm{a},s} = \pair{D+\phi(D),E+\phi(E)}_s =  \pair{D+\phi(D),E}_s \, .  \]
It is straightforward to see that the local N\'eron height pairing relative to $s$ is symmetric and bi-additive and depends only on the restrictions of $D,E$ to the generic fiber of $X \to S$. The pairing coincides with S. Zhang's admissible pairing \cite{zh} restricted to degree zero divisors.
\begin{thm} \label{main} Assume that the supports of $D,E$ are generically disjoint. Let $t$ be a uniformiser on $\overline{S}$ at $s$. Then the asymptotic relation
\[ g_{D_p}[E_p] \, \sim \, \pair{D,E}_{\mathrm{a},s} \log \abs{t(p)} \]
holds as $p \to s$ in the complex topology on $\overline{S}$. Here the notation $\sim$ means that the difference between left and right hand side extends as a bounded continuous function over a small open disc in $\overline{S}$ centered at $s$.
\end{thm}
Our approach to proving Theorem \ref{main} will be to use the Deligne pairing $\pair{D,E}$ between the divisors $D,E$ on $S$, to be discussed in Section \ref{section:deligne} below. This is a $C^\infty$-hermitian line bundle on $S$. We will derive Theorem \ref{main} from the following, more intrinsic result.
\begin{thm} \label{cor} Let $\langle D,E \rangle$ be Deligne's pairing on $S$ associated to the restrictions of $D,E$ to $X$. Let $m,n$ be positive integers such that $m\phi(D)=\phi(mD)$ and $n\phi(E)=\phi(nE)$ are divisors with integral coefficients on $\overline{X}$. Then the $C^\infty$-hermitian line bundle $\langle D,E \rangle^{\otimes mn}=\langle mD,nE\rangle$ has a unique extension  as a continuous hermitian line bundle over $\overline{S}$. The underlying line bundle of this extension is the Deligne pairing $\langle mD+\phi(mD),nE+\phi(nE)\rangle$ on $\overline{S}$ associated to the divisors $mD+\phi(mD)$ and $nE+\phi(nE)$ on $\overline{X}$.
\end{thm}
As an application of this result, we prove a special case of a conjecture concerning `height jumping' formulated recently \cite[Section~14]{hainnormal} by R. Hain. Let $\mathbb{U}$ denote a variation of polarized Hodge structure of weight~$-1$ over a smooth connected complex quasi-projective variety $Y$. Let $\mathcal{J}(\mathbb{U}) \to Y$ denote the corresponding intermediate jacobian fibration over $Y$. Denote by $\check{\mathbb{U}}$ the variation of Hodge structure dual to $\mathbb{U}$. Then as explained in e.g. \cite[Section~3]{hainbiext} or \cite[Section~6]{hainnormal} the torus fibration $\jj(\uu) \times_Y \jj(\check{\uu})$ over $Y$ carries a natural Poincar\'e (biextension) line bundle $\bb$, which comes equipped with a canonical $C^\infty$-hermitian metric. This metric has, among others, the following properties: its first Chern form is translation-invariant in all fibers of $\jj(\uu) \times_Y \jj(\check{\uu})$ over $Y$, and its pullback along the zero-section is trivial.

The polarization of $\uu$ gives rise to an isogeny $\lambda \colon \jj(\uu) \to \jj(\check{\uu})$ over $Y$; we denote by $\hat{\bb}_\lambda$ the pullback of the line bundle $\bb$ along the map $\jj(\uu) \to \jj(\uu) \times_Y \jj(\check{\uu})$ over $Y$ given by $(\mathrm{id},\lambda)$. Note that by pullback $\hat{\bb}_\lambda$ becomes equipped with a canonical $C^\infty$-hermitian metric. Let $\nu \colon Y \to \mathcal{J}(\mathbb{U})$ be a normal function section, and consider the $C^\infty$-hermitian line bundle $\mathcal{L}=\nu^*{\hat{\bb}_\lambda}$ on $Y$. A natural question is whether $\mathcal{L}$, or at least some tensor power of $\mathcal{L}$, can be extended as a continuous hermitian line bundle over a compactification of $Y$. A positive answer to this question has been given by D. Lear \cite{lear}, see \cite[Corollary 6.2]{hainnormal} for the formulation below.
\begin{thm} (D. Lear)  Let $\overline{Y}$ be a smooth compactification of $Y$ such that $\Delta=\overline{Y}-Y$ is a normal crossings divisor. Then there exists a positive integer $n$ such that the hermitian line bundle $\mathcal{L}^{\otimes n}$ extends as a line bundle with continuous hermitian metric over $\overline{Y} \setminus \Delta^{\mathrm{sing}}$. In particular, as $\codim_{\overline{Y}}\Delta^{\mathrm{sing}} \geq 2$, the underlying line bundle $\mathcal{L}^{\otimes n}$ has a canonical extension as a line bundle over $\overline{Y}$.
\end{thm}
For the sake of exposition we will usually assume that $\mathcal{L}$ itself extends (otherwise replace $\mathcal{L}$ by a high enough tensor power), and we will denote the resulting `Lear extension' over $\overline{Y}$ by $\mathcal{L}_{\overline{Y}}$. Now let $S$ denote a smooth connected complex algebraic curve, and let $f \colon S \to Y$ be a morphism. The data $\jj(\uu)$, $\hat{\bb}_\lambda$, $\nu$ and hence $\mathcal{L}$ pull back to $S$ along $f$. Let $\overline{S}$ denote a smooth compactification of $S$, and let $\overline{f} \colon \overline{S} \to \overline{Y}$ extend $f$.

As R. Hain notes in \cite[Section 14]{hainnormal}, perhaps surprisingly the formation of the Lear extension is not always compatible with pullback along $f$. More precisely, let $\overline{f}^* \mathcal{L}_{\overline{Y}}$ be the pullback of the Lear extension of $\mathcal{L}$ over $\overline{Y}$ along $\overline{f}$, and let $(f^*\mathcal{L})_{\overline{S}}$ be the Lear extension of $f^*\mathcal{L}$ over $\overline{S}$. The `difference' $\overline{f}^* \mathcal{L}_{\overline{Y}}\otimes (f^*\mathcal{L})^{\otimes -1}_{\overline{S}}$ is canonically trivial over $S$, and hence can be written, up to isomorphism, as $\mathcal{O}_{\overline{S}}(J)$ for a canonical divisor $J$ - called the \emph{jumping divisor} - on $\overline{S}$ supported on the boundary divisor $\overline{S} \setminus S$. Note that by the uniqueness of the Lear extension the height jumping divisor is trivial over $\overline{f}^*(\overline{Y} \setminus \Delta^{\mathrm{sing}})$. Thus, any non-trivial height jumping is `caused' by the singularities of the boundary divisor in $\overline{Y}$.

Let $\mm_{g,n}$ denote the moduli orbifold of smooth proper connected $n$-pointed complex curves of genus $g \geq 1$. Let $\uu$ be the natural polarised variation of Hodge structure on $\mm_{g,n}$ whose fiber at $[C,(x_1,\ldots,x_n)]$ is $\mathrm{H}_1(C,\zz)$. Note that $\jj(\uu)$ is then the pullback to $\mm_{g,n}$ of the universal jacobian over $\mm_g$.
In \cite[Section~14]{hainnormal} Hain conjectures that the height jumping divisor should be effective in the following two cases: (a) the inclusion $Y \subset \overline{Y}$ is the Deligne-Mumford compactification $\mm_{g,1} \subset \overline{\mm}_{g,1}$ ($g \geq 2$), and the normal function on $\mm_{g,1}$ is the function $\mathcal{K} \colon \mm_{g,1} \to \jj(\mathbb{U})$ given by sending $[C,x] \in \mm_{g,1}$ to the point $[\mathrm{Jac}(C),(2g-2)[x]-K_C]\in \jj(\mathbb{U})$, where $K_C$ is the canonical divisor class on $C$; (b) the inclusion $Y \subset \overline{Y}$ is the Deligne-Mumford compactification $\mm_g \subset \overline{\mm}_g$ ($g \geq 3)$, and the normal function on $\mm_g$ is the normal function $\nu \colon \mm_g \to \jj(\wedge^3 \uu)$ associated to the Ceresa cycle $C-C^-$ in the jacobian of $[C] \in \mm_g$.

Our next result implies Hain's conjecture for case (a). Let $d=(d_1,\ldots,d_n)$ be an $n$-tuple of integers and $m$ be an integer such that $\sum_i d_i = (2g-2)m$. Denote by $F_d$ the normal function $\mm_{g,n} \to \jj(\mathbb{U})$ given by sending $[C,(x_1,\ldots,x_n)]$ to $[\mathrm{Jac}(C),\sum_i d_i[x_i] -mK_C]$.
\begin{thm} \label{hainconj} Let $g\geq2$, $n \geq 1$ be integers. Let $S$ denote a smooth connected complex curve with smooth compactification $\overline{S}$. Let $\overline{f} \colon \overline{S} \to \overline{\mm}_{g,n}$ be a morphism, and assume that the restriction $f$ of $\overline{f}$ to $S$ has image contained in $\mm_{g,n}$. Then the height jumping divisor on $\overline{S}$ with respect to $f \colon S \to \mm_{g,n}$ and the normal function $F_d$, is effective.
\end{thm}
Our proof uses Theorem \ref{cor} as well as an interpretation of the non-archimedean N\'eron height pairing in terms of Green's functions and effective resistance on the reduction graph $\Gamma$ of a semistable curve. The fact that the Green's function on the set of divisors of degree zero on $\Gamma$ is positive definite (see Proposition~\ref{positive} below) will be tantamount to the required effectivity of the height jumping divisor.

We note that a result of Hain \cite[Theorem 12.3]{torelli} implies that for $g\geq 3$ and $n \geq 1$, each normal function section of $\jj(\uu) \to \mm_{g,n}$ is of the form $F_d$ for suitable $(d_1,\ldots,d_n)$ and $m$.

The organisation of this paper is as follows. In Sections \ref{section:poincare} and \ref{section:deligne} we state some preliminaries regarding Poincar\'e line bundles on families of principally polarized abelian varieties, and Deligne pairings. In Section \ref{section:cor} we prove Theorem \ref{cor}, and Theorem \ref{main} is derived from Theorem \ref{cor} in Section \ref{section:main}. Section \ref{section:graph} discusses the connection between the local non-archimedean N\'eron height pairing and the Green's function on the reduction graph. In Section \ref{section:conj} we then prove Theorem \ref{hainconj}. Finally, as a by-product of our method, we present in Section \ref{section:calc} an alternative derivation of Hain's expression \cite[Theorem 11.5]{hainnormal} for the Lear extension of $\mathcal{L}=F_{d}^* \hat{\bb}_\lambda$ from $\mm_{g,n}$ over $\overline{\mm}_{g,n}$.

\section{Poincar\'e line bundle} \label{section:poincare}

The purpose of this section is to describe the canonical $C^\infty$-hermitian metric on the Poincar\'e line bundle on a family of principally polarized abelian varieties. Referring to L. Moret-Bailly's article \cite{mb}, we next make a connection with the canonical metric on the determinant of cohomology, in the case where the family is a family of jacobians over a curve.

Let $Y$ be a connected complex manifold and let $\pi \colon \mathcal{T} \to Y$ be a family of complex tori over $Y$. Let $\hat{\mathcal{T}} \to Y$ be the torus fibration dual to $\mathcal{T} \to Y$, and let $\bb$ be the Poincar\'e bundle on the fiber product $\mathcal{T} \times_Y \hat{\mathcal{T}}$. Note that $\bb$ carries a canonical rigidification along the zero sections. A general construction described in \cite[Section~3.2]{hainbiext} (see also \cite[Section~7]{hrar}) yields the following result.
\begin{prop} \label{metriconbiext}  The Poincar\'e bundle $\bb$ carries a canonical $C^\infty$-hermitian metric $\|\cdot \|_\bb$. The metric $\|\cdot\|_\bb$ has the following two properties: (a) the first Chern form of $(\bb,\|\cdot\|_\bb)$ is translation-invariant in each of the fibers of $\tt \times_Y \hat{\tt} \to Y$; (b) the canonical rigidification of $\bb$ along the zero section is an isometry, where $\mathcal{O}_Y$ has the standard euclidean metric.
\end{prop}
We would like to make the canonical metric explicit in the case where the family $\tt \to Y$ is a family of principally polarized abelian varieties. It will suffice to consider the case where $Y=\aa_g$, the moduli orbifold of principally polarized abelian varieties, and $\tt \to Y$ is the universal family. In fact we will work with $Y=\hh_g$, the Siegel upper half space of degree $g$, with $\tt \to Y$ the universal abelian variety of dimension $g$, and show that the resulting expression for the canonical metric is $\mathrm{Sp}(2g,\zz)$-invariant. We view $\hh_g$ as the set of complex symmetric $g$-by-$g$ matrices with positive definite imaginary part. Then $\tt$ can be written as the quotient $(\zz^g \times \zz^g) \setminus (\cc^g \times \hh_g)$, where the action of $\zz^g \times \zz^g$ on $\cc^g \times \hh_g$ is defined by
\[ (m,n) \cdot (z;\tau) = (z+m+\tau n;\tau) \, . \]
Let $\Theta$ be the divisor on $\tt$ given by the standard Riemann theta function
\begin{equation} \label{Riemanntheta}
 \theta(z;\tau) = \sum_{n \in \zz^g} \exp(\pi i{}^t n \tau n + 2\pi i {}^t n z) \end{equation}
on the uniformization $\cc^g \times \hh_g$ of $\tt$. The line bundle $\oo(\Theta)$ has a natural $C^\infty$-hermitian metric $\|\cdot\|$ which is given as follows. Put $\|\theta\|=\|1\|_{\oo(\Theta)}$. Then
\begin{equation} \label{thetanorm} \|\theta\|(z;\tau) = (\det \Im \tau)^{1/4} \exp(-\pi {}^t \Im z (\Im \tau)^{-1} \Im z) |\theta(z;\tau)| \, .
\end{equation}
It can be verified that the expression on the right is invariant under the action of $(\zz^g \times \zz^g) \rtimes \mathrm{Sp}(2g,\zz)$.

The restriction of $\oo(\Theta)$ to a fiber of $\tt \to \hh_g$ is a symmetric ample line bundle that yields the canonical principal polarization of that fiber. The principal polarization induces an isomorphism $\lambda \colon \tt \to \hat{\tt}$; we write $\bb_\lambda$ for the pullback line bundle $(\mathrm{id},\lambda)^*\bb$ on $\tt \times_Y \tt$. By pulling back the metric on $\bb$ along $(\mathrm{id},\lambda)$ the line bundle $\bb_\lambda$ becomes equipped with a canonical $C^\infty$-hermitian metric $\|\cdot\|_{\bb_\lambda}$. This metric can be made explicit in the following way.

First, the analytic manifold $\hh_g$ carries a natural non-trivial $(1,1)$-form $\kappa$, the K\"ahler form of the canonical $\mathrm{Sp}(2g,\rr)$-invariant metric induced by the structure of $\hh_g$ as the locally symmetric variety $\mathrm{Sp}(2g,\rr)/U(g)$. From \cite[Chapter III]{si} we deduce that $\kappa$ is given by
\[ \kappa = \frac{i}{16 \pi} \mathrm{Tr}( (\Im \tau)^{-1} \cdot d\tau \cdot (\Im \tau)^{-1} \cdot d\overline{\tau} ) \, . \]
Analogously $\cc^g \times \hh_g$ has a natural  $(1,1)$-form $\mu$ given by
\[ \mu = \frac{i}{2} {}^t (dz -d\tau \cdot (\Im \tau)^{-1} \cdot \Im z) \cdot (\Im \tau)^{-1} \cdot (d \overline{z} - d \overline{\tau} \cdot (\Im \tau)^{-1} \cdot \Im z) \, , \]
which is $\zz^g \times \zz^g$-invariant and hence descends to a $(1,1)$-form on $\tt$. One has that $\mu$ is translation-invariant in each fiber, and $\mu$ vanishes along the zero section $e \colon Y \to \tt$. A computation yields that
\[ c_1 (\oo(\Theta),\|\cdot\|) = \mu + \pi^* \kappa \, . \]
In particular, the first Chern form of $(\oo(\Theta),\|\cdot\|)$ is translation-invariant in each fiber of $\tt \to \hh_g$.

Let $m \colon \tt \times_Y \tt \to \tt$ be the addition morphism, and $p_1,p_2 \colon \tt \times_Y \tt \to \tt$ the projections on the first and second factor, respectively. Denote by $\Lambda(\Theta)$ the line bundle
\[ m^* \oo(\Theta) \otimes p_1^* \oo(\Theta)^{\otimes -1} \otimes p_2^* \oo(\Theta)^{\otimes -1} \otimes e^* \oo(\Theta) \]
on $\tt \times_Y \tt$. We note that $\Lambda(\Theta)$ has a canonical rigidification along the zero section $(e,e) \colon Y \to \tt \times_Y \tt$. Also we note that $\Lambda(\Theta)$ has a natural $C^\infty$-hermitian metric induced from $(\oo(\Theta),\|\cdot\|)$. It follows directly from the properties mentioned above that the first Chern form of $\Lambda(\Theta)$ with its natural metric is translation-invariant in each fiber, and that the canonical rigidification of $\Lambda(\Theta)$ along $(e,e)$ is an isometry.
\begin{prop} \label{PoincareLambda} There exists a canonical isomorphism
\[ \psi \colon \bb_\lambda \isom \Lambda(\Theta)  \]
of line bundles on $\tt \times_Y \tt$ compatible with the canonical rigidifications along the zero sections on both sides.
\end{prop}
\begin{proof} Let $t \in Y$ be a point. According to \cite[Proposition 2.4.1]{bl} the points of the dual torus $\hat{\tt}_t$ parametrize the isomorphism classes of topologically trival line bundles on $\tt_t$. Write $\mathcal{L}=\oo(\Theta)$. The principal polarization $\lambda$ sends a point $y \in \tt_t$ to the class of the line bundle $T_y^*\mathcal{L}_t \otimes \mathcal{L}^{-1}_t$ in $\hat{\tt}_t$. Here $T_y$ denotes translation along $y$. The Poincar\'e bundle $\bb$ is the universal line bundle on $\tt \times_Y \hat{\tt}$, rigidified along the zero section \cite[Proposition 2.5.2]{bl}, and this gives us, for points $x,y \in \tt_t$, canonical identifications of fibers
\[ \begin{split} \bb_\lambda|_{(x,y)} & = \bb|_{(x,[T_y^*\mathcal{L} \otimes \mathcal{L}^{-1}])} \\ & = (T_y^*\mathcal{L} \otimes \mathcal{L}^{-1})|_x \otimes (T_y^*\mathcal{L} \otimes \mathcal{L}^{-1} )^{-1}|_e \\
& = \mathcal{L}|_{x+y} \otimes \mathcal{L}^{-1}|_x \otimes \mathcal{L}^{-1}|_y \otimes \mathcal{L}|_e \\
& = \Lambda(\Theta)|_{(x,y)} \, ,
\end{split} \]
which proves the proposition.
\end{proof}
By pullback along $\psi$ one thus obtains a $C^\infty$-hermitian metric on the line bundle $\bb_\lambda$ which has translation-invariant first Chern form in each fiber, and for which the canonical rigidification along the zero section is an isometry. As these two properties uniquely characterize a $C^\infty$-hermitian metric on $\bb_\lambda$ on $\tt \times_Y \tt$ by \cite[Corollary 5.4]{hrar}, we obtain that the metric thus constructed is in fact the canonical metric.

One immediately derives from (\ref{thetanorm}) the following explicit formula.
\begin{prop} \label{explicitnorm} Denote by $\eta$ the meromorphic section
\[ m^* \theta \otimes p_1^* \theta^{\otimes -1} \otimes p_2^* \theta^{\otimes -1} \otimes e^* \theta \]
of $\Lambda(\Theta)$. Then the canonical norm of $\eta$ is given by
\[ \|\eta\|_{\Lambda(\Theta)}(z,w;\tau) = \left| \frac{\theta(z+w;\tau)\theta(0;\tau) }{\theta(z;\tau)\theta(w;\tau)} \right| \exp(-2\pi {}^t (\Im z) (\Im \tau)^{-1} (\Im w) ) \]
for each $(z,w;\tau)$ in $\cc^g \times \cc^g \times \hh_g$.
\end{prop}
Assume for the moment that $Y$ is a connected Riemann surface, and let $\overline{Y}$ be the smooth compactification of $Y$. Let $\overline{\tt} \to \overline{Y}$ be the identity component of the N\'eron model of $\tt \to Y$.
\begin{prop} \label{existBbar} There exists a unique (up to isomorphism) rigidified line bundle $\overline{\bb}_\lambda$ on $\overline{\tt} \times_{\overline{Y}} \overline{\tt}$ extending the rigidified line bundle $\bb_\lambda$ on $\tt \times_Y \tt$.
\end{prop}
\begin{proof} See \cite[Proposition 2.8.2]{mb}.
\end{proof}
Assume that $Y$ is a smooth complex quasi-projective variety and let $p \colon Z \to Y$ be a smooth proper morphism with connected fibers of dimension one. Let $\jj \to Y$ be the jacobian fibration associated to the family of curves $Z \to Y$. Denote by $\check{\jj} \to Y$ the corresponding family of dual varieties. We have by \cite[Section~2.6]{mb} a canonical principal polarization $\lambda \colon \jj \isom \check{\jj}$. Let $\bb$ be the Poincar\'e bundle on $\jj \times_Y \check{\jj}$, and denote by $\bb_\lambda$ the $C^\infty$-hermitian rigidified line bundle $(\mathrm{id},\lambda)^*\bb$ on $\jj \times_Y \jj$.

Let $(\mathcal{L}, \mathcal{M})$ be a pair of line bundles of relative degree zero on $Z \to Y$. They naturally give rise to a normal function section $\nu \colon Y \to \jj \times_Y \jj$ by computing fiberwise the classes of $\mathcal{L}$ resp. $\mathcal{M}$ in the jacobian. Next, for any line bundle $\mathcal{L}$ of relative degree zero on $Y$, we denote by $\det \R p_* (\mathcal{L})$ the determinant of cohomology \cite{de} of $\mathcal{L}$  along $p$. This is a line bundle on $Y$.
\begin{prop} \label{detcohombiext} The tensor product
\[ \det \R p_*(\mathcal{L}\otimes\mathcal{M})^{\otimes -1} \otimes \det \R p_*(\mathcal{L}) \otimes \det \R p_*(\mathcal{M}) \otimes \det \R p_* (\mathcal{O})^{\otimes -1} \]
is canonically equipped with a $C^\infty$-hermitian metric.
Further, there exists a canonical isomorphism of line bundles
\[ \nu^* \bb_\lambda \isom \det \R p_*(\mathcal{L}\otimes\mathcal{M})^{\otimes -1} \otimes \det \R p_*(\mathcal{L}) \otimes \det \R p_*(\mathcal{M}) \otimes \det \R p_* (\mathcal{O})^{\otimes -1} \]
on $Y$ which is an isometry for the canonical metrics on left and right hand side.
\end{prop}
\begin{proof} Let $X$ be a compact Riemann surface of positive genus $g$. \cite[Th\'eor\`eme 4.13 and Remarque 4.13.1(b)]{mb} imply the existence, for each line bundle $\mathcal{L}$ of degree zero on $X$ equipped with a hermitian metric $\|\cdot \|$ with vanishing first Chern form, of a canonical metric on the tensor product of determinants of cohomology $\det \R \Gamma(X,\mathcal{L}) \otimes \det \R \Gamma(X,\mathcal{O}_X)^{\otimes -1}$. Changing the metric $\|\cdot \|$ on $\mathcal{L}$ to a metric $\alpha \|\cdot\|$ for some $\alpha \in \rr_{>0}$ will change the canonical metric on $\det \R \Gamma(X,\mathcal{L}) \otimes \det \R \Gamma(X,\mathcal{O}_X)^{\otimes -1}$ by a factor $\alpha^{\chi(\mathcal{L})}=\alpha^{1-g}$. It follows that for flat hermitian line bundles $\mathcal{L}$, $\mathcal{M}$, the tensor product of determinants of cohomology
\[ \det \R \Gamma(\mathcal{L}\otimes\mathcal{M})^{\otimes -1} \otimes \det \R \Gamma(\mathcal{L}) \otimes \det \R \Gamma(\mathcal{M}) \otimes \det \R \Gamma (\mathcal{O})^{\otimes -1} \]
is equipped with a canonical hermitian metric independent of the flat metrics chosen on $\mathcal{L}$, $\mathcal{M}$. Let $\bb_\lambda$ on $J \times J$ be the pullback along $(\mathrm{id},\lambda)$ of the Poincar\'e bundle on $J\times \check{J}$ where $J$ is the jacobian of $X$, and $\lambda \colon J \isom \check{J}$ the canonical principal polarisation. Then \cite[Corollaire~4.14.1]{mb} asserts that there exists a canonical isomorphism
\[  (\mathcal{L},\mathcal{M})^*\bb_\lambda \isom \det \R \Gamma(\mathcal{L}\otimes\mathcal{M})^{\otimes -1} \otimes \det \R \Gamma(\mathcal{L}) \otimes \det \R \Gamma(\mathcal{M}) \otimes \det \R \Gamma (\mathcal{O})^{\otimes -1} \]
of $\cc$-vector spaces which is an isometry for the canonical metrics on left and right hand side. Letting $X$ vary in the family $Z \to Y$ we obtain the isometry from the proposition. As the metric on the left hand side varies in a $C^\infty$ way, so does the canonical metric on the right hand side.
\end{proof}
Now assume that $\dim(Y)=1$, and let $\overline{Y}$ be the smooth compactification of $Y$. Let $\overline{Z} \to \overline{Y}$ be a proper flat morphism extending $Z \to Y$, and assume that (a) $\overline{Z}$ is smooth and (b) the fibers of $\overline{Z} \to \overline{Y}$ are reduced and have only ordinary double points as singularities. Let $\overline{\jj} \to \overline{Y}$ be the identity component of the N\'eron model of $\jj \to Y$ over $\overline{Y}$. By Proposition \ref{existBbar} there exists a unique (up to isomorphism) rigidified line bundle $\overline{\bb}_\lambda$ on $\overline{\jj} \times_{\overline{Y}} \overline{\jj}$ extending the rigidified line bundle $\bb_\lambda$ on $\jj \times_Y \jj$.

Now let $\mathcal{L},\mathcal{M}$ be two line bundles on $\overline{Z}$, of degree zero on each irreducible component of each fiber of $\overline{Z} \to\overline{Y}$. By a result of Raynaud \cite[Theorem 9.7.1]{blr} one has a natural isomorphism between $\overline{\jj}$ and the identity component $\Pic^0(\overline{Z}/\overline{Y})$ of the Picard scheme of $\overline{Z}/\overline{Y}$. Thus, by computing fiberwise the classes of $\mathcal{L}$ resp. $\mathcal{M}$ in $\Pic^0(\overline{Z}/\overline{Y})$ we obtain a natural normal function section $\overline{\nu} \colon \overline{Y} \to \overline{\jj} \times_{\overline{Y}} \overline{\jj}$ extending the normal function $\nu \colon Y \to \jj \times_Y \jj$ associated to the restrictions of $\mathcal{L}$, $\mathcal{M}$ to $Z \to Y$.
\begin{prop} \label{detcohombiextbar} We have a canonical isomorphism
\[ \overline{\nu}^* \overline{\bb}_\lambda \isom  \det \R p_*(\mathcal{L}\otimes\mathcal{M})^{\otimes -1} \otimes \det \R p_*(\mathcal{L}) \otimes \det \R p_*(\mathcal{M}) \otimes \det \R p_* (\mathcal{O})^{\otimes -1} \]
of line bundles on $\overline{Y}$ extending the isomorphism from Proposition \ref{detcohombiext} on $Y$.
\end{prop}
\begin{proof} See \cite[Corollaire 2.8.5]{mb}.
\end{proof}

\section{Deligne pairing} \label{section:deligne}

In this section we recall Deligne's pairing and its connection with the Poincar\'e line bundle. References for this section are \cite[Section XIII.5]{acg} and \cite[Section~6]{de}. Let $Y$ be a smooth connected complex algebraic variety and let $p \colon Z \to Y$ be a smooth proper morphism with connected fibers of dimension one. Let $\mathcal{L}, \mathcal{M}$ be two line bundles on $Z$. Then to these data one has canonically associated a line bundle $\langle \mathcal{L},\mathcal{M} \rangle$ on $Y$, as follows: local generators are symbols $\langle l,m \rangle$, where $l,m$ are local generating sections of $\mathcal{L}, \mathcal{M}$. The relations to be satisfied are
\[ \langle l, fm \rangle = f[\divisor{l}] \cdot \langle l,m \rangle \, , \quad
\langle fl, m \rangle = f[\divisor{m}] \cdot \langle l, m \rangle  \]
for all local regular functions $f$. We call $\pair{\mathcal{L},\mathcal{M}}$ the Deligne pairing of $\mathcal{L},\mathcal{M}$. It is straightforward to verify that for line bundles $\mathcal{L}_1$, $\mathcal{L}_2$, $\mathcal{M}_1$, $\mathcal{M}_2$, $\mathcal{L}$, $\mathcal{M}$ on $X$ we have canonical isomorphisms \[ \langle \mathcal{L}_1 \otimes \mathcal{L}_2, \mathcal{M} \rangle \isom \langle \mathcal{L}_1 , \mathcal{M} \rangle \otimes \langle
\mathcal{L}_2 , \mathcal{M} \rangle \, , \, \langle \mathcal{L}, \mathcal{M}_1 \otimes \mathcal{M}_2 \rangle \isom \langle \mathcal{L}, \mathcal{M}_1 \rangle \otimes \langle \mathcal{L} , \mathcal{M}_2 \rangle \, , \]
and $\langle \mathcal{L},\mathcal{M} \rangle \isom \langle \mathcal{M} , \mathcal{L}\rangle$. An isomorphism $\mathcal{L}_1 \isom \mathcal{L}_2$ induces a natural isomorphism $\langle \mathcal{L}_1, \mathcal{M} \rangle \isom \langle \mathcal{L}_2, \mathcal{M} \rangle$. Moreover, the formation of the Deligne pairing commutes with arbitrary base change.

Now assume $\mathcal{L}, \mathcal{M}$ are equipped with $C^\infty$-hermitian metrics. Then we can put a natural hermitian structure $\|\cdot\|$ on $\langle \mathcal{L} , \mathcal{M} \rangle$ by requiring (cf. \cite[Section~6.3]{de})
\begin{equation} \label{metric} \log \| \langle l,m \rangle \| = \int_{Z/Y} \left(\frac{\deldelbar}{\pi i} \log \| l \| \right) \cdot \log \| m \| + \log \| l \| [\divisor{m}] + \log \| m \| [\divisor{l}]
\end{equation}
for local generating sections $l$ resp. $m$ with disjoint supports, where the $\deldelbar$ is taken in the sense of distributions. It turns out that $\|\cdot \|$ is a $C^\infty$-hermitian metric on $\langle \mathcal{L},\mathcal{M} \rangle$, and each of the above canonical isomorphisms is isometric.

We are in particular interested in the case $\mathcal{L}=\mathcal{O}_Z(D)$, $\mathcal{M}=\mathcal{O}_Z(E)$ where $D,E$ are divisors on $Z$ of relative degree zero. Then we write $\langle D,E \rangle$ as a shorthand for $\langle \mathcal{O}_Z(D), \mathcal{O}_Z(E)\rangle $. Let $1_D$ resp. $1_E$ denote the canonical meromorphic sections of $\mathcal{O}_Z(D)$ resp. $\mathcal{O}_Z(E)$. They give rise to a canonical meromorphic section $\langle 1_D,1_E \rangle$ of the Deligne pairing $\langle D,E \rangle$.

We can put a natural $C^\infty$-hermitian structure on $\mathcal{O}_Z(D)$ and $\mathcal{O}_Z(E)$ as follows. Let $g$ be the genus of the fibers of $p \colon Z \to Y$. Let $\mathcal{F}$ be the vector bundle $p_* \Omega^1_{Z/Y}$ of rank $g$ on $Y$. Note first of all that the flat intersection form on the local system $\mathrm{R}^1 p_* \zz_Z(1)$ of weight $-1$ extends to a flat non-degenerate $\rr$-bilinear alternating pairing $E$ on the dual bundle $\mathcal{F}^\lor$. Let $H \colon (v,w) \mapsto E(iv,w)+iE(v,w)$ be the corresponding hermitian form; then $H$ is well known to be positive definite.

Let $H^*$ be the induced hermitian form on $\mathcal{F}$ (given by $H^*(\omega,\omega') = \frac{i}{2} \int_{Z/Y} \omega \wedge \bar{\omega}'$).
Let $(\omega_1,\ldots,\omega_g)$ be a local holomorphic frame of $\mathcal{F}$, and let $B$ be the matrix of $H^*$ with respect to this frame. We then have a canonical $(1,1)$-form $\mu$ on $Z$ by putting locally
\[ \mu = \frac{i}{2g} \sum_{j,k=1}^g    \overline{B}^{-1}_{j,k} \, \omega_j \bar{\omega}_k \, , \]
which is independent of the chosen frame. On each fiber of $Z \to Y$, the form $\mu$ restricts to the canonical (Arakelov) volume form \cite[Section~3]{mb} \cite[Section~2]{we}.

Now using $\mu$ one has a natural way of normalizing the Green's function $g_{D_p}$ associated to the restriction $D_p$ of $D$ to $Z_p$ by requiring that
\begin{equation} \label{normalisation} \int_{Z/Y} g_{D} \, \mu = 0
\end{equation}
identically on $Y$. Assume the Green's functions $g_{D_p}$ are normalised this way. The line bundles $\mathcal{O}_X(D), \mathcal{O}_X(E)$ then have a canonical $C^\infty$-hermitian structure, given by putting $\log \|1_D\|(q) = g_{D_p}(q)$ for $q \in Z_p$ outside the support of $D_p$, and likewise $\log \|1_E\|(q) = g_{E_p}(q)$ for $q \in Z_p$ outside the support of $E_p$.

The resulting hermitian structure on $\pair{D,E}$ can be characterised as follows. Assume $D,E$ have generically disjoint support. Let $V$ be a non-empty open subset of $Y$ such that $D_p,E_p$ are divisors with disjoint support on $X_p$ for each $p \in V$. Then the restriction of $\langle 1_D,1_E \rangle$ to $V$ is a generating section of $\langle D,E \rangle $ over $V$.  By equations (\ref{defproperty}) and (\ref{normalisation}), the metric (\ref{metric}) on $\langle D,E \rangle $ is just given by the formula
\begin{equation} \label{formulalognorm}
\log \| \langle 1_D,1_E \rangle \| = g_{D_p}[E_p]
\end{equation}
for $p \in V$. That is, we have a natural interpretation of the archimedean N\'eron height pairing as the log norm of a canonical section in a suitable Deligne pairing on $Y$.

As in Section \ref{section:poincare} let $\det \R p_* (\mathcal{L})$ denote the determinant of cohomology of a line bundle $\mathcal{L}$ on $Z$ along~$p$. Recall from Proposition \ref{detcohombiext} that the tensor product
\[ \det \R p_*(\mathcal{L}\otimes\mathcal{M})^{\otimes -1} \otimes \det \R p_*(\mathcal{L}) \otimes \det \R p_*(\mathcal{M}) \otimes \det \R p_* (\mathcal{O})^{\otimes -1} \]
is equipped with a canonical $C^\infty$-hermitian metric. The metrized version of the Riemann-Roch theorem proven in \cite[Th\'eor\`eme 11.4]{de} now implies the following result.
\begin{prop} \label{rr} Assume that both $\mathcal{L}, \mathcal{M}$ have relative degree zero. There exists a canonical isomorphism of line bundles
\begin{equation} \label{rreqn} \pair{\mathcal{L},\mathcal{M}} \isom \det \R p_*(\mathcal{L}\otimes\mathcal{M}) \otimes \det \R p_*(\mathcal{L})^{\otimes -1} \otimes \det \R p_*(\mathcal{M})^{\otimes -1} \otimes \det \R p_* (\mathcal{O})
\end{equation}
on $Y$. This isomorphism becomes an isometry upon equipping left and right hand side with their canonical metrics.
\end{prop}
Let $\jj \to Y$ be the jacobian fibration associated to the morphism $Z \to Y$, let $\bb$ be the Poincar\'e bundle on $\jj \times_Y \check{\jj}$, let $\lambda \colon \jj \isom \check{\jj}$ be the canonical polarization, and write $\bb_\lambda$ for the $C^\infty$-hermitian rigidified line bundle $(\mathrm{id},\lambda)^*\bb$ on $\jj \times_Y \jj$. Consider again a pair $\mathcal{L}$, $\mathcal{M}$ of line bundles of relative degree zero on $Z \to Y$. Let $\nu \colon Y \to \jj \times_Y \jj$ be the normal function section given by computing fiberwise the classes of $\mathcal{L}$ resp. $\mathcal{M}$ in the jacobian.
\begin{cor} \label{pairingbiext} There exists a canonical isomorphism of line bundles
\begin{equation} \label{pb} \pair{\mathcal{L},\mathcal{M}}^{\otimes -1} \isom \nu^* \bb_\lambda
\end{equation}
on $Y$. This isomorphism becomes an isometry upon equipping left and right hand side with their canonical metrics.
\end{cor}
\begin{proof} This follows immediately upon combining Propositions \ref{detcohombiext} and \ref{rr}.
\end{proof}
Finally, let $\overline{Y}$ be a smooth compactification of $Y$. Let $\overline{Z} \to \overline{Y}$ be a proper flat morphism extending $Z \to Y$, and assume as before that (a) $\overline{Z}$ is smooth and (b) the fibers of $\overline{Z} \to \overline{Y}$ are reduced and have only ordinary double points as singularities. Let $\mathcal{L},\mathcal{M}$ be two line bundles on $\overline{Z}$. The construction outlined above of the Deligne pairing of $\mathcal{L}$, $\mathcal{M}$ as a line bundle using generators and relations carries over to the present situation \cite[Section XIII.5]{acg} and yields a line bundle $\pair{\mathcal{L}, \mathcal{M}}$ on $\overline{Y}$. Likewise, one has the determinant of cohomology $\det \R p_*\mathcal{L}$ which is a line bundle on $\overline{Y}$.

Assume that $\dim(Y)=1$. Let $\overline{\jj} \to \overline{Y}$ be the identity component of the N\'eron model of the jacobian fibration $\jj \to Y$ over $\overline{Y}$. Let $\overline{\bb}_\lambda$ be the unique rigidified line bundle on $\overline{\jj} \times_{\overline{Y}} \overline{\jj}$ extending the line bundle $\bb_\lambda$ on $\jj \times_Y \jj$, whose existence is guaranteed by Proposition \ref{existBbar}.

Assume that both $\mathcal{L}$ and $\mathcal{M}$ are of degree zero on each irreducible component of each fiber of $\overline{Z} \to\overline{Y}$. Let $\overline{\nu} \colon \overline{\jj} \times_{\overline{Y}} \overline{\jj}$ be the normal function extending the normal function $\nu \colon Y \to \jj \times_Y \jj$ associated to the restrictions of $\mathcal{L}$, $\mathcal{M}$ to $Z \to Y$, as discussed in the previous section.
\begin{prop} \label{pairingbiextbar} There exists a canonical isomorphism of line bundles
\[  \pair{\mathcal{L},\mathcal{M}}^{\otimes -1} \isom \overline{\nu}^* \overline{\bb}_\lambda \]
on $\overline{Y}$, extending the isomorphism from (\ref{pb}) on $Y$.
\end{prop}
\begin{proof} According to \cite[Theorem XIII.5.8]{acg} we have a canonical isomorphism
\[ \pair{\mathcal{L},\mathcal{M}} \isom \det \R p_*(\mathcal{L}\otimes\mathcal{M}) \otimes \det \R p_*(\mathcal{L})^{\otimes -1} \otimes \det \R p_*(\mathcal{M})^{\otimes -1} \otimes \det \R p_* (\mathcal{O}) \]
of line bundles on $\overline{Y}$ extending the isomorphism (\ref{rreqn}) on $Y$. The proposition follows by applying Proposition \ref{detcohombiextbar}.
\end{proof}

\section{Proof of Theorem \ref{cor}} \label{section:cor}

Let $Y$ be a smooth connected complex curve, and $\overline{Y}$ its smooth compactification. Let $(\tt \to Y,\lambda)$ be a family of principally polarized abelian varieties. Assume that the monodromy around each boundary point is unipotent. By \cite[Expos\'e XI, 3.5--3.8]{SGA7}, the identity component  $\overline{\tt} \to \overline{Y}$ of the N\'eron model of $\tt \to Y$ over $\overline{Y}$ is a family of semiabelian varieties. As in Section \ref{section:poincare} we have the Poincar\'e bundle $\bb_\lambda$ on $\tt \times_Y \tt$ and a canonical extension $\overline{\bb}_\lambda$ on $\overline{\tt} \times_{\overline{Y}} \overline{\tt}$. Both $\bb_\lambda$ and its extension $\overline{\bb}_\lambda$ are endowed with a rigidification along the zero section; in fact this property characterizes the extension $\overline{\bb}_\lambda$ of $\bb_\lambda$. Recall that $\bb_\lambda$ is endowed with a canonical $C^\infty$-hermitian metric, which is made explicit in Proposition \ref{explicitnorm}, and the rigidification of $\bb_\lambda$ along the zero section is an isometry for the standard euclidean metric on $\oo_Y$.

We will derive Theorem \ref{cor} from the following general result. The result is a special case of \cite[Proposition 6.1]{lear}, which unfortunately has not been published.
\begin{thm} \label{learprop6.1} The canonical $C^\infty$-hermitian metric on $\bb_\lambda$ extends in a continuous manner over $\overline{\bb}_\lambda$.
\end{thm}
\begin{proof} The question is local for the analytic topology on $\overline{Y}$ so we may replace $Y$ by the punctured unit disk $\Delta^*$ and $\overline{Y}$ by the unit disk $\Delta$. Let $\hh$ be the Siegel upper half plane and let $\hh \to \Delta^*, v \mapsto t=\exp(2\pi i v)$ denote the universal cover of $\Delta^*$. Write $g$ for the relative dimension of the family $\tt \to \Delta^*$. Let $\hh_g$ denote the Siegel upper half space of degree $g$. As the monodromy around the origin is unipotent by assumption, we may assume that the period map $j \colon \hh \to \hh_g$ associated to the family $\tt \to \Delta^*$ satisfies the condition $j(v+1)=j(v) + A$ for some integral symmetric matrix $A$. According to \cite[Lemma 2.3]{nak}, as a multi-valued map on $\Delta^*$ the period map $j$ can in fact be written as
\begin{equation} \label{period}
j(t) = \frac{A}{2\pi i} \log t + B(t) \quad \textrm{for all} \quad t \in \Delta^* \, , \end{equation}
for some bounded holomorphic matrix $B \colon \Delta \to M_{g,g}(\cc)$ over $\Delta$. Here, the matrix $A$ is integral, symmetric and positive semi-definite. Let $\UU \to \hh_g$ be the universal abelian variety. The multi-valued map $j$ induces a single-valued map $\bar{j} \colon \Delta^* \to \langle A \rangle \setminus \hh_g$ and the family $\tt \to \Delta^*$ can be obtained from the family $\pair{A} \setminus \UU \to \pair{A} \setminus \hh_g$ by base change along $\bar{j}$. After replacing $\Delta^*$ by a finite cover $\Delta^* \to \Delta^*$ we may assume that the matrix $A$ is even. The expression for Riemann's theta function given in (\ref{Riemanntheta}) is then invariant under the monodromy action $\tau \mapsto \tau + A$. In particular, the divisor $\Theta$ and line bundle $\oo(\Theta)$ descend to a divisor and line bundle on $\pair{A} \setminus \UU$, which we continue to denote by the same symbols. The natural projection $\cc^g \times \hh_g \to \pair{A} \setminus \UU$ pulls back along $\bar{j}$ to a projection $\cc^g \times \Delta^* \to \tt$. By applying $j^*$ to the expression in (\ref{Riemanntheta}), we obtain a pullback theta function $j^*\theta$ on $\cc^g \times \Delta^*$ as well as a pullback divisor $j^*\Theta$ and a pullback line bundle $j^*\oo(\Theta)$ on $\tt$. Similarly, the $C^\infty$-hermitian line bundle $\Lambda(\Theta)$ defined in Section \ref{section:poincare} pulls back to a $C^\infty$-hermitian line bundle $j^*\Lambda(\Theta)$ on $\tt \times_{\Delta^*} \tt$. The line bundle $j^*\Lambda(\Theta)$ is rigidified along the zero section, and the rigidification is an isometry. By Proposition \ref{PoincareLambda} we have an isometry $\bb_\lambda \isom j^* \Lambda(\Theta)$.  Applying (\ref{period}) we obtain the series expansion
\begin{equation} \label{expansion} j^*\theta  =  \sum_{n \in \zz^g} t^{\frac{1}{2} {}^t n A n} \exp(\pi i {}^t n B(t) n + 2\pi i {}^t n z)  \end{equation}
for the pullback theta function on $\cc^g \times \Delta^*$.  As $A$ is positive semi-definite, each summand in this expansion is holomorphic. Let $M$ be a real number. After shrinking $\Delta$ if necessary we can find $c_1, c_2 >0$ such that
\[  |t^{\frac{1}{2} {}^t n A n} \exp(\pi i {}^t n B(t) n + 2\pi i {}^t n z)| < c_1 \exp(-c_2 \|n\|_\infty^2)  \]
for each $n \in \zz^g$, each $t \in \Delta$ and each $z \in \cc^g$ with $\|z\|_\infty \leq M$. It follows that the series in (\ref{expansion}) converges absolutely and uniformly on compact subsets of $\cc^g \times \Delta$, and consequently $j^*\theta$ extends as a holomorphic function on $\cc^g \times \Delta$. Note that the uniformization $\cc^g \times \Delta^* \to \tt$ extends into a uniformization $\cc^g \times \Delta \to \overline{\tt}$.  Write $\overline{\Theta}$ for the divisor induced by $j^*\theta$ on $\overline{\tt}$, and $\oo(\overline{\Theta})$ for the associated line bundle over $\overline{\tt}$. Let $m \colon \overline{\tt} \times_\Delta \overline{\tt} \to \overline{\tt}$ be the group law of the semi-abelian family $\overline{\tt} \to \Delta$, let $p_1,p_2 \colon \overline{\tt} \times_\Delta \overline{\tt} \to \overline{\tt}$ be the projections on the first resp. second factor, and let $e$ be the zero section of $\overline{\tt} \to \Delta$.
Put
\[ \overline{\Lambda} = m^* \oo(\overline{\Theta}) \otimes p_1^* \oo(\overline{\Theta})^{\otimes{-1}} \otimes p_2^* \oo(\overline{\Theta})^{\otimes -1}
\otimes e^* \oo(\overline{\Theta}) \]
on $\overline{\tt} \times_\Delta \overline{\tt}$.
Then the line bundle $\overline{\Lambda}$ is canonically rigidified along the zero section, and hence there is a natural isomorphism of holomorphic line bundles $\overline{\bb}_\lambda \isom \overline{\Lambda}$ extending the isometry $\bb_\lambda \isom j^*\Lambda(\Theta)$. We will be done once we prove that the metric on $j^*\Lambda(\Theta)$ extends in a continuous manner over $\overline{\Lambda}$.
Denote by $\eta$ the meromorphic section $m^* \theta \otimes p_1^* \theta^{\otimes{-1}} \otimes p_2^* \theta^{\otimes -1}
\otimes e^* \theta$ of $\overline{\Lambda}$. Put $t_0=0$.
Let $s$ be a local generating section of $\overline{\Lambda}$ around a point $(z_0,w_0;t_0)  \in \overline{\tt} \times_\Delta \overline{\tt}$. Then there is a meromorphic function $f$ on $\overline{\tt} \times_\Delta \overline{\tt}$ such that the identities
\[ s(z,w;t) = f(z,w;t) \eta(z,w;t) = f(z,w;t) \frac{ \theta(z+w;j(t))\theta(0;j(t)) }{\theta(z;j(t))\theta(w;j(t))} \]
hold locally around $(z_0,w_0;t_0)$. By Proposition \ref{explicitnorm} we find
\[ \begin{split} \| s\|(z,w;t) & = |f|(z,w;t) \|\eta\|(z,w;t) \\ & = \left| f(z,w;t) \frac{ \theta(z+w;j(t))\theta(0;j(t)) }{\theta(z;j(t))\theta(w;j(t))} \right| \exp(-2\pi {}^t (\Im z) (\Im j(t) )^{-1} (\Im w)) \end{split} \]
for $t\in \Delta^*$. We will be done once we prove that the function $\|s\|(z,w;t)$ extends as a non-vanishing continuous function over $t_0=0$. First of all, as $s$ is generating we find that the factor
\[ \left| f(z,w;t) \frac{ \theta(z+w;j(t))\theta(0;j(t)) }{\theta(z;j(t))\theta(w;j(t))} \right|  \]
extends in a non-vanishing continuous manner over $t_0=0$. It remains to consider the exponential factor
\[ \exp(-2\pi {}^t (\Im z) (\Im j(t) )^{-1} (\Im w)) \, . \]
It suffices to prove that the function $(\Im j(t))^{-1}$ extends continuously over $\Delta$. Write $u(t)=-\frac{1}{2\pi} \log |t|$ for $t \in \Delta^*$. From
equation (\ref{period}) we obtain that
\[ \Im j(t) = A\cdot u(t) + \Im B(t) \]
for $t \in \Delta^*$. As $A$ is integral, symmetric and positive semi-definite there exists an orthogonal matrix $C$ such that $CAC^{-1}$ has the shape
\[ CAC^{-1} = \left( \begin{array}{c|c}
0 & 0 \\ \hline  0 & \renewcommand{\arraystretch}{0.1} \begin{array}{ccc}
\lambda_1 & & 0 \\
 & \ddots & \\
0 & & \lambda_n \end{array} \end{array} \right) \]
with $\lambda_1,\ldots,\lambda_n$ the positive eigenvalues of $A$. Write
\[ C(\Im B(t)) C^{-1} = \left( \begin{array}{c|c}
Q_1(t) & Q_2(t) \\ \hline  {}^t Q_2(t) & Q_3(t) \end{array} \right) \quad \textrm{and} \quad D = \left( \renewcommand{\arraystretch}{0.1} \begin{array}{ccc}
\lambda_1 & & 0 \\
 & \ddots & \\
0 & & \lambda_n \end{array} \right) \, \]
where $Q_1(t)$ is a $(g-n)\times(g-n)$-matrix and $Q_3(t)$ is an $n \times n$-matrix.
Then
\[ \Im j(t) = C^{-1} \cdot \left( \begin{array}{c|c}
Q_1(t) & Q_2(t) \\ \hline  {}^t Q_2(t) & D \cdot u(t) + Q_3(t) \end{array} \right) \cdot C \]
and hence it suffices to prove that the inverse of the positive definite matrix
\[ R(t) = \left( \begin{array}{c|c}
Q_1(t) & Q_2(t) \\ \hline  {}^t Q_2(t) & D \cdot u(t) + Q_3(t) \end{array} \right) \]
extends continuously over $\Delta$. Note that each $Q_i(t)$ is a bounded continuous matrix on $\Delta$. Moreover, by \cite[Section 4.4.4]{chai}, the matrix $Q_1(0)$ can be viewed as a period matrix of the abelian variety part of the semiabelian variety $\overline{\tt}_0$ and hence is positive definite. Write $\lambda=\lambda_1\cdots \lambda_n \in \rr_{>0}$. We see that $\det R(t)$ is a polynomial of degree $n$ in $u(t)$ with bounded continuous coefficients and with leading coefficient $\lambda \cdot \det Q_1(t)$ which stays away from zero as $t \to 0$. Further, each cofactor of $R(t)$ is a polynomial of degree $\leq n$ in $u(t)$ with bounded continuous coefficients. Let $\gamma(t)$ be such a cofactor and write
\[ \gamma(t) = a_m(t)u(t)^m + \cdots + a_0(t) \, , \quad
\det R(t) = b_n(t) u(t)^n + \cdots + b_0(t) \, , \]
with $m \leq n$, with $a_i(t), b_j(t)$ bounded continuous functions and $b_n(t)=\lambda \cdot \det Q_1(t)$. Then the entries of $R(t)^{-1}$ have the shape
\[ \frac{ a_m(t) u(t)^{-(n-m)} + \cdots +a_0(t) u(t)^{-n} }
{ b_n(t) + b_{n-1}(t)u(t)^{-1} + \cdots +b_0(t) u(t)^{-n} } \, . \]
Now note that $u(t)^{-1}$ extends continuously over $\Delta$ with value $0$ at $t=0$. Also recall that $b_n(t)$ is continuous and does not vanish at $t=0$. We deduce that each entry of $R(t)^{-1}$ extends continuously over $\Delta$, and we are done.
\end{proof}
\begin{remark} The shape of the exponential factor $\sim \exp(-c(\log|t|)^{-1})$ indicates that in general, the canonical metric on $\bb_\lambda$ does not extend in a $C^\infty$ manner over $\overline{\bb}_\lambda$.
\end{remark}
\begin{proof}[Proof of Theorem \ref{cor}]
Let $\overline{\jj} \to \overline{S}$ be the identity component of the N\'eron model of the jacobian family $\jj \to S$ of $X \to S$ over $\overline{S}$. We have the Poincar\'e bundle $\bb_\lambda$ over $\jj \times_S \jj$ and its canonical extension $\overline{\bb}_\lambda$ over $\overline{\jj} \times_{\overline{S}} \overline{\jj}$ characterized by the property that the rigidification along the zero section of $\bb_\lambda$ extends over $\overline{\bb}_\lambda$. We have $\overline{\jj} \isom \Pic^0(\overline{X}/\overline{S})$ by Raynaud's theorem \cite[Theorem 9.7.1]{blr}. Hence the choice of the integers $m,n$ implies that the sections $\mu_1,\mu_2$ of $\jj \to S$ given by $\mu_1=[mD]$, $\mu_2=[nE]$ extend as sections $\bar{\mu}_1, \bar{\mu}_2$ of $\overline{\jj} \to \overline{S}$. More precisely, by definition of $\phi$ we can write $\bar{\mu}_1=[mD+\phi(mD)]$ and $\bar{\mu}_2=[nE+\phi(nE)]$, upon identifying $\overline{\jj}$ with $\Pic^0(\overline{X}/\overline{S})$.
As the family $\overline{X} \to \overline{S}$ is semistable, we find that the monodromy of $\overline{\jj} \to \overline{S}$ around each boundary point is unipotent. Applying Theorem \ref{learprop6.1} we obtain that $\overline{\bb}_\lambda$ can be endowed with a continuous hermitian metric extending the canonical metric on $\bb_\lambda$. By specializing we find that $(\bar{\mu}_1,\bar{\mu}_2)^* \overline{\bb}_\lambda$ can be endowed with a continuous hermitian metric extending the pullback metric on $(\mu_1,\mu_2)^* \bb_\lambda$.
By Proposition \ref{pairingbiext} we have a canonical isometry
\begin{equation} \label{canonicalisometry}
(\mu_1,\mu_2)^* \bb_\lambda \isom \pair{mD,nE} = \pair{D,E}^{\otimes mn}
\end{equation}
of $C^\infty$-hermitian line bundles on $S$. By Proposition \ref{pairingbiextbar} we have a canonical isomorphism
\[ (\bar{\mu}_1,\bar{\mu}_2)^* \overline{\bb}_\lambda \isom \pair{mD+\phi(mD),nE+\phi(nE)} \]
of line bundles on $\overline{S}$ extending the isomorphism (\ref{canonicalisometry}). Combining these observations the theorem follows.
\end{proof}

\section{Proof of Theorem \ref{main}} \label{section:main}

In this section we give a proof of Theorem \ref{main}. We deduce the result from the following elementary lemma, whose proof is left to the reader.
\begin{lem} \label{asympt} Let $\mathcal{L}$ be a holomorphic line bundle on $\Delta^*$ equipped with a continuous hermitian metric $\|\cdot\|_\mathcal{L}$. Let $\overline{\mathcal{L}}$ be an extension of $\mathcal{L}$ over $\Delta$, and assume that the metric $\|\cdot\|_{\mathcal{L}}$ extends in a continuous fashion over $\overline{\mathcal{L}}$. Let $s$ be a generating section of $\mathcal{L}$ over $\Delta^*$, and view $s$ as a meromorphic section of $\overline{\mathcal{L}}$. Let $a$ be its multiplicity at~$0$. Then the asymptotic relation $\log \|s\|_\mathcal{L}(p) \sim a \log |t(p)|$ holds as $p \to 0$ on $\Delta^*$.
\end{lem}
The proof of Theorem \ref{main} follows essentially by observing that the Green's functions $g_{D_p}[E_p]$ define the canonical metric on $\pair{D,E}$.
\begin{proof}[Proof of Theorem \ref{main}] By Theorem \ref{cor} the line bundle $\pair{mD+\phi(mD),nE+\phi(nE)}$ on $\overline{S}$ carries a continuous hermitian metric that extends the canonical hermitian metric $\| \cdot \|$ on $\pair{mD,nE}$ over $S$. Let $V \subset \overline{S}$ be an open neighborhood of $s$, such that the supports of $D,E$ are disjoint over $V \cap S$. Then over $V \cap S$, the canonical meromorphic section $\langle 1_{mD},1_{nE}\rangle$ of $\langle mD,nE \rangle $ is generating.
By equation (\ref{formulalognorm}) we have for the canonical metric $\| \cdot \|$ on $\langle mD, nE \rangle $ that $\log \| \langle 1_{mD},1_{nE}\rangle \|(p) = g_{mD_p}[nE_p]$ for each $p$ in $V$. By Lemma \ref{asympt} we are done once we prove that when we view $\pair{1_{mD},1_{nE}}$ as a meromorphic section of $\pair{mD+\phi(mD),nE+\phi(nE)}$ we have $\ord_s\pair{1_{mD},1_{nE}}=\pair{mD,nE}_{\textrm{a},s}$. Note that
the canonical meromorphic section $\langle 1_{mD+\phi(mD)}, 1_{nE+\phi(nE)} \rangle$ of $\pair{mD+\phi(mD),nE+\phi(nE)}$ extends the canonical meromorphic section $\langle 1_{mD},1_{nE} \rangle$ of $\langle mD,nE \rangle$ over $\overline{S}$. This gives
\[ \begin{split} \ord_s\langle 1_{mD},1_{nE} \rangle & = \ord_s \langle 1_{mD+\phi(mD)}, 1_{nE+\phi(nE)} \rangle \\ & = \langle mD+\phi(mD),nE+\phi(nE) \rangle_s = \pair{mD,nE}_{\textrm{a},s} \, , \end{split} \]
and we are done.
\end{proof}

\section{N\'eron pairing and Green's function on the reduction graph} \label{section:graph}

The purpose of this section is to relate the pairing $\pair{D,E}_{\textrm{a},s}$ to the Green's function on the reduction graph of the fiber $F$ of $\overline{X}$ at $s$. The material in this section is probably well known, and certainly implicit in \cite{zh}, but we have chosen to include proofs for results for which we do not know a suitable reference.

Let $\Gamma$ be a finite connected graph. Let $T_0$ be the set of vertices and $T_1$ the set of edges of $\Gamma$. Choose an orientation on $\Gamma$. This gives rise to the usual source and target maps $s,t \colon T_1\to T_0$. Let $d_* = t_*-s_* \colon \qq^{T_1} \to \qq^{T_0}$ and $d^* = t^* - s^* \colon \qq^{T_0} \to \qq^{T_1}$ be the boundary and coboundary maps. The \emph{Laplacian} $L$ of $\Gamma$ is defined to be the standard matrix of the composite $d_*d^* \colon \qq^{T_0} \to \qq^{T_0}$. The matrix $L$ is symmetric, and independent of the choice of orientation. We call the elements of $\qq^{T_0}$ \emph{divisors} on $\Gamma$. For a divisor $\dd= \sum_C a_C C$ we call $\sum_C a_C \in \qq$ the \emph{degree} of $\dd$.

As $\Gamma$ is connected, the kernel of $L$ consists of the constant functions, and the image of $L$ is the orthogonal complement of the constant functions, i.e. the set of divisors of degree zero. Let $L^+$ be the Moore-Penrose pseudo-inverse of $L$, i.e. the unique symmetric matrix of the same size as $L$ satisfying the two conditions $LL^+L=L$, $L^+LL^+=L^+$. We view $L^+$ as a bilinear pairing on the set $\qq^{T_0}$ of divisors. We call the $\emph{Green's function}$ on $\Gamma$ the function $g_\Gamma \colon T_0 \times T_0 \to \qq$ which associates to a pair of vertices $(C,C')$ of $\Gamma$ the rational number $L^+(\delta_C,\delta_{C'})$.

The following connection with electric network theory is useful. Thinking of each edge of $\Gamma$ as a resistance of one unit, we may identify $\Gamma$ with an electric circuit. Let $I \in \qq^{T_0}$ denote a degree-zero function which we think of as an electric flow through the circuit $\Gamma$ that has $I(C)$ units of current entering the circuit at a vertex $C$, if $I(C) > 0$, and $-I(C)$ units of current leaving the circuit at $C$, if $I(C) <0$. We look for a function $P \in \qq^{T_0}$ giving voltages that realize this flow. Such a function $P$ is determined by the matrix equation $LP=I$, hence such a $P$ exists, and is uniquely defined up to adding a constant function. A special solution is given by $P=L^+I$. Indeed, we can write $I=LI'$ for some $I'\in \qq^{T_0}$ as $I$ has degree zero, and then $LL^+I=LL^+LI'=LI'=I$.

Next, for any two vertices $C,C'$ of $\Gamma$ one has the \emph{effective resistance} $r_\Gamma(C,C') \in \qq$ between $C,C'$, defined as follows. Consider an electric flow that sends one unit of electric current into vertex $C$ and removes one unit of electric current from vertex $C'$. The effective resistance $r_\Gamma(C,C')$ between $C,C'$ is then defined to be the potential difference between $C$ and $C'$ that is required to realize this flow. In terms of the above, let $I \in \qq^{T_0}$ be the degree-zero divisor that has value $+1$ at $C$, $-1$ at $C'$, and $0$ else. Let $P=L^+I$. Then we have the useful relation
\begin{equation} \label{useful} r_\Gamma(C,C')=P(C)-P(C')={}^tIP={}^tI L^+I = g_\Gamma(C,C)-2g_\Gamma(C,C')+g_\Gamma(C',C')
\end{equation}
between the Green's function, on the one hand, and the effective resistance on the other.

By linearity we extend the effective resistance function as a bi-additive $\qq$-valued pairing on $\qq^{T_0}$. More precisely, whenever $\dd=\sum_i a_i C_i$ and $\ee=\sum_j b_jC_j$ are divisors on $\Gamma$, we write $r_\Gamma(\dd,\ee)=\sum_{i,j} a_ib_j r_\Gamma(C_i,C_j)$. Analogously we put $g_\Gamma(\dd,\ee)=\sum_{i,j} a_ib_j g_\Gamma(C_i,C_j)$.
\begin{prop} \label{green_r} Assume both $\dd,\ee$ are degree-zero divisors on $\Gamma$. Then the formula
\[  g_\Gamma(\dd,\ee) = -\frac{1}{2}r_\Gamma(\dd,\ee) \, .
\]
holds.
\end{prop}
\begin{proof} It suffices by linearity to verify the case that $\dd=C_1-C_2$, $\ee=C_1'-C_2'$, where $C_1$, $C_2$, $C_1'$ and $C_2'$ are vertices of $\Gamma$. Then equation (\ref{useful}) yields
\[ \begin{split}
r_\Gamma(C_1-C_2,C_1'-C_2') &= r_\Gamma(C_1,C_1')-r_\Gamma(C_1,C_2')-r_\Gamma(C_2,C_1')+r_\Gamma(C_2,C_2') \\
& = -2(g_\Gamma(C_1,C_1')-g_\Gamma(C_1,C_2')-g_\Gamma(C_2,C_1')+g_\Gamma(C_2,C_2')) \, ,
\end{split} \]
and we are done.
\end{proof}
Let $q$ be a designated element of $\qq^{T_0}$. One then defines the \emph{canonical divisor} of $(\Gamma,q)$ to be the divisor $K=\sum_{C \in T_0} (v(C)+2q(C)-2)C$ of $\Gamma$. Here $v(C)$ denotes the valence of $C$. The \emph{genus} of $(\Gamma,q)$ is defined to be the integer $g=1+\frac{1}{2}\deg K = b(\Gamma)+\sum_C q(C)$, where $b(\Gamma) \in \zz$ denotes the Betti number of $\Gamma$. We say that $q$ is a \emph{polarisation}, and $(\Gamma,q)$ a \emph{polarised graph}, if $q$ is non-negative.

Next we view $\Gamma$ as a connected topological space by identifying each edge with a closed interval. We call a subgraph $\Gamma'$ of $\Gamma$ a \emph{bridge} if $\Gamma'$ has precisely one edge $e$, and if upon removing the interior of $e$ a non-connected space results. We call a subgraph $\Gamma'$ of $\Gamma$ a \emph{$2$-connected component} of $\Gamma$ if for each point $p$ on $\Gamma'$, the topological space $\Gamma' \setminus \{p\}$ is connected. We can write $\Gamma$ as a successive finite pointed sum of subgraphs $\Gamma_i$ such that each $\Gamma_i$ is either a bridge or a $2$-connected component of $\Gamma$.

Assume we are given such a finite pointed sum, then for each $i$ we have a natural projection map $\pi_i \colon \Gamma \to \Gamma_i$. The pushforward $\pi_{i*}$ along $\pi_i$ of divisors is well-defined, and preserves degrees. In particular, if $(\Gamma,q)$ is a polarised graph, we naturally obtain polarised graphs $(\Gamma_i,q_i)$ with canonical divisors $K_i$. Each $(\Gamma_i,q_i)$ has the same genus as $(\Gamma,q)$.

We have the following additivity property, which often proves useful in computations.
\begin{prop} \label{additivity}
Write $\Gamma=\sum_i \Gamma_i$ as the successive pointed sum of its bridges and $2$-connected components. Let $g_{\Gamma_i}$ denote the Green's function on the subgraph $\Gamma_i$. Assume $\dd,\ee$ are divisors of degree zero on $\Gamma$. Write $\dd_i=\pi_{i*}(\dd)$ and $\ee_i=\pi_{i*}(\ee)$. Then the equality
\[ g_\Gamma(\dd,\ee) = \sum_i g_{\Gamma_i}(\dd_i,\ee_i) \]
holds in $\qq$.
\end{prop}
\begin{proof} By Proposition \ref{green_r} it suffices to prove that $r_\Gamma(C,C')=\sum_i r_{\Gamma_i}(\pi_i(C),\pi_i(C'))$ for vertices $C,C'$ of $\Gamma$. This can be deduced by induction on the number of bridges and $2$-connected components $\Gamma_i$ from the following property, which follows easily from the interpretation of $\Gamma$ as an electric circuit: let $C,C',C''$ be vertices of $\Gamma$, let $\Gamma', \Gamma''$ be subgraphs such that $\Gamma$ is the pointed sum of $\Gamma',\Gamma''$ along $C=\Gamma'\cap \Gamma''$. Assume $C' \in \Gamma'$ and $C'' \in \Gamma''$. Then $r_\Gamma(C',C'')=r_{\Gamma'}(C,C')+r_{\Gamma''}(C,C'')$.
\end{proof}
The following result will be crucial for Hain's conjecture.
\begin{prop} \label{positive} When viewed as a symmetric bilinear pairing on the set of divisors of degree zero on $\Gamma$, the Green's function $g_\Gamma$ is positive definite.
\end{prop}
\begin{proof} As $\Gamma$ is connected, the Laplacian $L$ is positive semi-definite with signature $(0+ \cdots +)$. Write $L=U\Sigma U^{-1}$ with $\Sigma$ a diagonal matrix. Then it is easily verified that $L^+=U\Sigma^+U^{-1}$, with $\Sigma^+$ the Moore-Penrose pseudo-inverse of $\Sigma$. Note that $\Sigma^+$ can be obtained from $\Sigma$ by replacing each non-zero diagonal entry by its reciprocal. It follows that $L^+$ has the same signature as $L$. As the constant divisors are isotropic for $L^+$, we find the required result.
\end{proof}
Now consider a proper flat morphism $\overline{X} \to \overline{S}$ with connected fibers of dimension one as before, with $\overline{S}$ a smooth proper complex curve, and $\overline{X}$ regular.  We continue to assume that the fibers of the morphism $\overline{X} \to \overline{S}$ are reduced and have only ordinary double points as singularities. Let $s$ be a closed point of $\overline{S}$, and write $F$ for the fiber of $\overline{X}\to \overline{S}$ at $s$. The \emph{reduction graph} $\Gamma$ of $F$ is then the finite graph given as follows: the set of vertices is the set of irreducible components of $F$, and an edge is drawn between two (possibly equal) vertices $C_1,C_2$ for each time the components $C_1,C_2$ intersect in a double point. The graph $\Gamma$ is connected since $F$ is connected. The Laplacian matrix of $\Gamma$ is equal to minus the intersection matrix $(\pair{C_i,C_j}_s)$ of $F$. Associating to each vertex $C$ the arithmetic genus of $C$ gives a polarisation $q$ on $\Gamma$.

Using the local intersection pairing above $s$, if $D,E$ are divisors on $\overline{X}$ we obtain divisors $\dd=\sum_C \pair{D,C}_s C$ and $\ee=\sum_C \pair{E,C}_s C$ on $\Gamma$, called the \emph{reductions} of $D,E$ on $\Gamma$, respectively. In particular, let $R$ be the relative dualising sheaf of $\overline{X} \to \overline{S}$, then by the adjunction formula the reduction of $R$ on $\Gamma$ is equal to the canonical divisor $K$ associated to the polarisation $q$. Recall that $\phi(D)$ is any $\qq$-divisor on $\overline{X}$ such that $D+\phi(D)$ has zero intersection product with all components of the fibers of $\overline{X}\to \overline{S}$.

We have the following connection with the Green's function on $\Gamma$.
\begin{prop} Assume both $D,E$ are divisors of relative degree zero on $\overline{X}/\overline{S}$, and consider their reductions $\dd=\sum_C \pair{D,C}_s C$ and $\ee=\sum_C \pair{E,C}_s C$ on $\Gamma$. Then the equality
\[ g_\Gamma(\dd,\ee)=-\pair{\phi(D),\phi(E)}_s \]
holds in $\qq$.
\end{prop}
\begin{proof} Recall that $L$ equals minus the intersection matrix of $F$.
We find that when viewing both the reduction $\dd$ of $D$ and $\phi(D)$ as elements of $\qq^{T_0}$ we have the matrix equation $L\cdot \phi(D) = -\dd$. We can thus take $\phi(D)=-L^+\dd$ and likewise $\phi(E)=-L^+\ee$.  Hence
\[
-\pair{\phi(D),\phi(E)}_s = {}^t\phi(D) L \phi(E) = {}^t \dd L^+LL^+ \ee
 = {}^t\dd L^+ \ee = g_\Gamma(\dd,\ee) \, ,
\]
as required.
\end{proof}
\begin{cor} \label{green_admissible} Assume both $D,E$ are divisors of relative degree zero on $\overline{X}$, and write $\dd=\sum_C \pair{D,C}_s C$ and $\ee=\sum_C \pair{E,C}_s C$ for their reductions on $\Gamma$. Then the equality
\[ \pair{D,E}_{\textrm{a},s} =\pair{D,E}_s+g_\Gamma(\dd,\ee) \]
holds in $\qq$.
\end{cor}
\begin{proof} This follows from the previous proposition by observing that \[ \pair{D,E}_{\textrm{a},s}=\pair{D,E}_s+\pair{\phi(D),E}_s = \pair{D,E}_s - \pair{\phi(D),\phi(E)}_s \, , \]
where the latter equality holds because $\pair{E+\phi(E),\phi(D)}_s=0$.
\end{proof}
Assume finally that $\overline{X} \to \overline{S}$ is endowed with $n$ disjoint sections $x_1,\ldots,x_n \colon \overline{S} \to \overline{X}$. Let $I=\{x_1,x_2,\ldots,x_n\}$. We have a natural map $u \colon I \to T_0$ by associating to each $x_i \in I$ the unique irreducible component $C$ of $F$ such that $x_i$ intersects $C$ non-trivially. Let $e$ a bridge of $\Gamma$. Upon removing the interior of $e$ we obtain two connected graphs $\Gamma_1,\Gamma_2$ out of $\Gamma$, which we consider with their induced polarisations. Let $h$ with $0 \leq h \leq g$ be the genus of $\Gamma_1$; by additivity we have that $g-h$ is the genus of $\Gamma_2$. Let $P \subseteq I$ be a subset, and write $P^c=I \setminus P$. We say that $e$ is of \emph{type} $(P,h)$ if $u(P)$ is contained in the vertex set of $\Gamma_1$, and $u(P^c)$ is contained in the vertex set of $\Gamma_2$. Note that a bridge of type $(P,h)$ is also of type $(P^c,g-h)$.

\section{Proof of Theorem \ref{hainconj}} \label{section:conj}

Let $g\geq2$, $n \geq 1$ be integers, and let $d=(d_1,\ldots,d_n),m$ be an $n$-tuple of integers resp. an integer such that $\sum_i d_i = (2g-2)m$. Let $Y=\mm_{g,n}$ be the moduli space of smooth proper connected $n$-pointed complex curves of genus $g$, and let $\overline{Y}=\overline{\mm}_{g,n}$ denote its Deligne-Mumford compactification. Let $Z \to Y$ resp. $\overline{Z} \to \overline{Y}$ be the corresponding universal curves. We view each of $Y$, $Z$, $\overline{Y}$, $\overline{Z}$ as orbifolds.

Note that $\overline{Z} \to \overline{Y}$ comes equipped with an $n$-tuple $(x_1,\ldots,x_n)$ of sections, whose images $[x_i]$ in $\overline{Z}$ can be viewed as divisors on $\overline{Z}$. Let $\omega_{\overline{Z}/\overline{Y}}$ be the relative dualising sheaf of $\overline{Z} \to \overline{Y}$, and let $R$ be a divisor on $\overline{Z}$ such that $\mathcal{O}_{\overline{Z}}(R)$ is isomorphic to $\omega_{\overline{Z}/\overline{Y}}$. We then define $D$ to be the divisor $\sum_i d_i [x_i] - mR$ on $\overline{Z}$. We have a well defined Deligne pairing $\pair{D,D}$ on $\overline{Y}$.

Let $\jj_{g,n}$ be the pullback of the universal jacobian over $\mm_g$ to $Y$. Let $\bb$ be the Poincar\'e line bundle on the fiber product $\jj_{g,n} \times_Y \check{\jj}_{g,n}$, where $\check{\jj}_{g,n}$ denotes the dual torus fibration. We are interested in the normal function section
\[ F_{d} \colon Y \to \jj_{g,n} \, , \quad [C,(x_1,\ldots,x_n)] \mapsto [\mathrm{Jac}(C),\sum_i d_i[x_i] -mK_C] \, , \]
where $K_C$ is the class of a canonical divisor on $C$. Note  that the Abel-Jacobi class of the restriction of $D$ to the fiber of $Z \to Y$ over a point $[C,(x_1,\ldots,x_n)]$ and the image of $[C,(x_1,\ldots,x_n)]$ in $\jj_{g,n}$ under the normal function section $F_{d}$, are equal.

Let $\lambda \colon \jj_{g,n} \isom \check{\jj}_{g,n}$ denote the canonical principal polarization, and let $\hat{\bb}_\lambda$ denote the $C^\infty$-hermitian line bundle $(\mathrm{id},\lambda)^*\bb$ on $\jj_{g,n}$. Theorem \ref{pairingbiext} immediately gives the following.
\begin{prop} \label{nerontate} Let $\langle D,D \rangle$ be the restriction to $Y$ of the Deligne pairing of $D$ with itself. Then there exists an isometry
\[ \langle D,D \rangle ^{\otimes -1} \isom F_{d}^* \hat{\bb}_\lambda \]
of $C^\infty$-hermitian line bundles on $Y$.
\end{prop}
Now let $S$ denote a smooth connected complex curve with smooth compactification $\overline{S}$, and consider a morphism $\overline{f} \colon \overline{S} \to \overline{Y}$ such that the restriction $f$ of $\overline{f}$ to $S$ has image contained in $Y$. Let $(X \to S,(x_1,\ldots,x_n))$ resp. $(\overline{X} \to \overline{S},(x_1,\ldots,x_n))$ be families of pointed smooth resp. pointed semistable curves of genus $g$ with classifying morphisms $f$ resp. $\overline{f}$. We will assume that $\overline{X}$ is smooth over $\cc$ and extends $X$. By pullback along $\overline{f}$ we have a natural divisor $D$ on $\overline{X}$ compatible with the one on $\overline{Z}$ defined above, as well as a Deligne pairing $\langle D,D \rangle$ on $\overline{S}$ compatible with the pullback along $\overline{f}$ of the Deligne pairing $\pair{D,D}$ on $\overline{Y}$.

Write $\mathcal{L}$ for the $C^\infty$-hermitian line bundle $F_{d}^*\hat{\bb}_\lambda$ on $Y$. Our aim is to compare the Lear extension of $f^*\mathcal{L}$ from $S$ over $\overline{S}$ with the pullback of the Lear extension of $\mathcal{L}$ from $Y$ over $\overline{Y}$. We start with the former. We always neglect the possibility that taking tensor powers is necessary.
\begin{prop} \label{lear_curve_prep} Let $\phi(D)$ be a $\qq$-divisor on $\overline{X}$ such that $D+\phi(D)$ has zero intersection product with each irreducible component of every fiber of $\overline{X} \to \overline{S}$. Then the Lear extension of $f^*\mathcal{L}$ from $S$ over $\overline{S}$ is isomorphic with $\pair{D+\phi(D),D+\phi(D)}^{\otimes -1}$.
\end{prop}
\begin{proof} It follows from Proposition \ref{nerontate} and compatibility with pullback along $f \colon S \to Y$ that we have an isometry $f^*\mathcal{L} \isom \pair{D,D}^{\otimes -1}$ of hermitian line bundles on $S$. The Lear extension is (neglecting the possiblility of needing to take tensor powers) the unique extension of $f^*\mathcal{L}$ as a continuous hermitian line bundle over $\overline{S}$. By Theorem \ref{cor} the Lear extension of $f^*\mathcal{L}$ is thus isomorphic to $\pair{D+\phi(D),D+\phi(D)}^{\otimes -1}$.
\end{proof}
From Corollary \ref{green_admissible} we then obtain
\begin{cor} \label{lear_curve} For each closed point $s \in \overline{S}$ let $\Gamma_s$ be the reduction graph of the fiber of $\overline{X} \to \overline{S}$ at $s$, and let $\dd_s$ be the reduction of $D$ to $\Gamma_s$. Write $a_s = g_{\Gamma_s}(\dd_s,\dd_s)$.
Then the class in $\mathrm{Pic}(\overline{S})$ of the Lear extension of $f^*\mathcal{L}$ from $S$ over $\overline{S}$ is equal to the class of  $-\pair{D,D}-\sum_s a_s s$.
\end{cor}
The Lear extension of $\mathcal{L}$ from $Y$ over $\overline{Y}$ is calculated in \cite[Theorem 11.5]{hainnormal}. We present here an alternative approach. We start with recalling the structure of the boundary divisor $\Delta$ of $\overline{Y}$, see e.g. \cite[Section~9]{hainnormal} whose notation we adopt. Each irreducible component of $\Delta$ has as generic point a stable $n$-pointed curve of genus $g$ with precisely one node. This leads to the following boundary components:
\begin{itemize}
\item $\Delta_0$: the generic point is an irreducible geometrically connected $n$-pointed stable curve with precisely one node;
\item $\Delta_0^P$, with $P$ a subset of the set $I=\{x_1,\ldots,x_n\}$ of marked points with $|P| \geq 2$: the generic point is a reducible stable curve with two geometrically connected components joined at a node, one of which has genus zero; the points in $P$ lie on the genus zero component minus the node, and the points in $P^c=I \setminus P$ lie on the other (genus $g$) component minus the node;
\item $\Delta_h^P$, where $1 \leq h \leq g-1$ and $P \subseteq I$ is a (possibly empty) subset of $I$: the generic point is a reducible stable curve with two geometrically connected components joined at a node, one of genus $h$, the other of genus $g-h$; the points in $P$ lie on the genus $h$ component minus the node, and the points in $P^c$ lie on the genus $g-h$ component minus the node.
\end{itemize}
We denote by $\delta_0$, $\delta_0^P$, $\delta_h^P$ the corresponding classes in $\mathrm{Pic}(\overline{Y})$. Note that $\delta_h^P=\delta_{g-h}^{P^c}$. Let $\pair{R,R}$ on $\overline{Y}$ be the Deligne pairing of $R$ with itself, and $\pair{[x_i],R}$ the Deligne pairing between $[x_i]$ and $R$. Then by \cite[Theorem 2]{ac} the group $\mathrm{Pic}(\overline{Y})$ is generated by the classes of
\[ \pair{R,R} \, , \quad \pair{[x_i],R} \, (i \in I) \, , \quad \delta_0 \, , \quad \delta_0^P \, (P \subseteq I \, , |P| \geq 2) \, , \quad \delta_h^P \, (1 \leq h \leq g/2 \, , P \subseteq I) \, , \]
with no linear relations between the $\delta$'s.

Consider a graph $\Gamma$ consisting of two vertices $C_1,C_2$, connected by one edge of unit length. Let $g_\Gamma$ denote its Green's function. First, let $P\subseteq I$ be a subset with $|P|\geq 2$. Write
\[ \dd_0^P(d) = (m+\sum_{i \in P}d_i)C_1+(-(2g-1)m+\sum_{i \in P^c}d_i)C_2 \, , \]
and define $a(P,d)=g_\Gamma(\dd_0^P(d),\dd_0^P(d))$. Next, let $P \subseteq I$ be an arbitrary subset and $h$ an integer with $1 \leq h \leq g-1$. Then write
\[ \dd_h^P(d)=(-(2h-1)m+\sum_{i \in P}d_i)C_1 + (-(2(g-h)-1)m+\sum_{i \in P^c}d_i)C_2\, , \]
and put $a(P,h,d)=g_\Gamma(\dd_h^P(d),\dd_h^P(d))$. A straightforward calculation (see also Section~\ref{section:calc} below) shows that $a(P,d)$ and $a(P,h,d)$ are in fact integers. Note moreover that $a(P,h,d)=a(P^c,g-h,d)$.
\begin{prop} \label{lear_moduli_space}
Let  $\mathcal{L}_{\overline{Y}}$ denote the Lear extension of $\mathcal{L}$ from $Y=\mm_{g,n}$ over $\overline{Y}=\overline{\mm}_{g,n}$. Then the equality
\[ \mathcal{L}_{\overline{Y}} = -\pair{D,D}-\sum_{P \subseteq I, |P| \geq 2} a(P,d)\delta_0^P - \frac{1}{2}\sum_{h=1}^{g-1} \sum_{P \subseteq I} a(P,h,d)\delta_h^P   \]
holds in $\mathrm{Pic}(\overline{\mm}_{g,n})$.
\end{prop}
\begin{proof} As $\mathcal{L}_{\overline{Y}}$ and $\pair{D,D}^{\otimes -1}$ coincide over $Y$ by Proposition \ref{nerontate}, there exist unique rational numbers $\alpha(0,d)$, $\alpha(P,d)$ and $\alpha(P,h,d)$ such that the identity
\[ \mathcal{L}_{\overline{Y}} = - \pair{D,D}-\alpha(0,d)\delta_0-\sum_{P \subseteq I, |P| \geq 2} \alpha(P,d)\delta_0^P - \frac{1}{2}\sum_{h=1}^{g-1} \sum_{P \subseteq I} \alpha(P,h,d)\delta_h^P   \]
holds in $\mathrm{Pic}(\overline{Y})$. Our task is to show that $\alpha(0,d)=0$, $\alpha(P,d)=a(P,d)$ and $\alpha(P,h,d)=a(P,h,d)$. We do this using the method of test curves. Let $\Delta_\bullet$ be an irreducible component of the boundary divisor $\Delta$ of $\overline{Y}$. Then the singular locus of $\Delta_\bullet$ is given by the set of classes of stable curves in $\Delta_\bullet$ with at least two nodes. Let $S$ be a smooth connected complex curve with smooth compactification $\overline{S}$ and let $\overline{f} \colon \overline{S} \to \overline{Y}$ be a test curve such that the generic point of $\overline{S}$ maps to $Y$, and the image intersects the boundary divisor $\Delta$ transversely in a non-singular point of $\Delta_\bullet$.

Consider first the case that $\Delta_\bullet \neq \Delta_0$. Let $-\alpha$ be the coefficient of the class of $\Delta_\bullet$ in $\mathcal{L}_{\overline{Y}} \otimes \pair{D,D}$. Let $f$ be the restriction of $\overline{f}$ to $S$. As the image of $\overline{f}$ does not intersect the singular locus of $\Delta$, the Lear extension of $f^*\mathcal{L}$ equals the pullback $\overline{f}^*(\mathcal{L}_{\overline{Y}})$ of the Lear extension of $\mathcal{L}$ over $\overline{Y}$. Let $\Gamma$ be the reduction graph of the stable curve $\overline{f}(s)$ and let $\dd$ be the reduction of $D$ to $\Gamma$.  By Corollary \ref{lear_curve} we find locally around $s$ that
\[ \mathcal{O}_{\overline{S}}(-\alpha s) = \mathcal{O}_{\overline{S}}(-\alpha \overline{f}^*\Delta_\bullet) = \overline{f}^*\mathcal{L}_{\overline{Y}} \otimes \pair{D,D} = (f^*\mathcal{L})_{\overline{S}} \otimes \pair{D,D} = \mathcal{O}_{\overline{S}}(-g_\Gamma(\dd,\dd)s) \]
hence $\alpha = g_\Gamma(\dd,\dd)$. Now $\Gamma$ consists of two vertices $C_1,C_2$ joined by one edge. Moreover if $\Delta_\bullet=\Delta_0^P$ resp. $\Delta_\bullet=\Delta_h^P$ the reduction of the divisor $\dd$ on $\Gamma$ precisely coincides with the divisor $\dd_0^P$ resp. $\dd_h^P$ on $\Gamma$. Hence we find
\[ \alpha(P,d) = g_\Gamma(\dd_0^P,\dd_0^P)=a(P,d) \quad \mathrm{and} \quad
\alpha(P,h,d) = g_\Gamma(\dd_h^P,\dd_h^P)=a(P,h,d)  \, . \]
It remains to show that $\alpha(0,d)=0$. In the case $\Delta_\bullet=\Delta_0$, the generic point of $\Delta_\bullet$ is an irreducible $n$-pointed curve with one node (with neither of the marked points passing through the node). The reduction graph $\Gamma$ has one vertex, and one edge, and the reduction $\dd$ of $D$ to $\Gamma$ is the empty divisor. We find $\alpha(0,d)=g_\Gamma(\dd,\dd)=0$ as required.
\end{proof}
\begin{proof}[Proof of Theorem \ref{hainconj}] As in the theorem let $S$ denote a smooth connected complex curve with smooth compactification $\overline{S}$, let $\overline{f} \colon \overline{S} \to \overline{Y}$ be a morphism such that the restriction $f$ of $\overline{f}$ to $S$ has image contained in $Y$, and let $X \to S$ and $\overline{X} \to \overline{S}$ be families of smooth resp. semistable curves of genus $g$ with classifying morphisms $f$ resp. $\overline{f}$. We assume that $\overline{X}$ is smooth over $\cc$ and extends $X$.

Let $J$ be the jumping divisor on $\overline{S}$ with respect to $f \colon S \to \mm_{g,n}$ and $F_{d}$, that is we have
\[ \mathcal{O}_{\overline{S}}(J) = \overline{f}^* \mathcal{L}_{\overline{Y}}\otimes (f^*\mathcal{L})^{\otimes -1}_{\overline{S}} \, . \]
Write $J=\sum j_s s$ with $j_s \in \qq$; note that we have $j_s=0$ for each $s\in S$. Our aim is to prove that $j_s\geq 0$ for each $s \in \overline{S}$.

Fix a closed point $s \in \overline{S}$. Let $\Gamma$ be the reduction graph of the fiber of $\overline{X} \to \overline{S}$ at $s$, and let $\dd$ be the reduction of $D$ to $\Gamma$. Recall that $\Gamma$ is naturally a polarised graph of genus $g$. Write $a = g_{\Gamma}(\dd,\dd)$. Let $\ell_0^P$ (for $P\subseteq I$ with $|P| \geq 2$) denote the number of edges of type $(0,P)$ in $\Gamma$, and let $\ell_h^P$ (for $1 \leq h \leq g-1$ and $P \subseteq I$) denote the number of edges of type $(h,P)$ in $\Gamma$. The geometric meaning of these numbers is that the local multiplicity at $s$ of $\overline{f}^*(\Delta_0^P)$ equals $\ell_0^P$, and the local multiplicity at $s$ of $\overline{f}^*(\Delta_h^P)$ equals $\ell_h^P$.
By combining Corollary \ref{lear_curve} and Proposition \ref{lear_moduli_space} we find that $j_s$ equals
\[ j_s=-\sum_{P \subseteq I, |P| \geq 2} a(P,d)\ell_0^P - \frac{1}{2} \sum_{h=1}^{g-1} \sum_{P \subseteq I} a(P,h,d)\ell_h^P + a \, . \]
We need to show, therefore, that the inequality
\[ a \geq \sum_{P \subseteq I, |P| \geq 2} a(P,d)\ell_0^P + \frac{1}{2} \sum_{h=1}^{g-1} \sum_{P \subseteq I} a(P,h,d)\ell_h^P \]
holds. We use the additivity relation from Proposition \ref{additivity}. Write $\Gamma=\sum_i \Gamma_i$ as the pointed sum of its bridges and $2$-connected components. We view each $\Gamma_i$ as a polarised graph with the polarisation induced from $\Gamma$. Let $g_{\Gamma_i}$ denote the Green's function on the subgraph $\Gamma_i$. Let $\pi \colon \Gamma \to \Gamma_i$ denote the natural projection map. For each $i$ write $\dd_i=\pi_{i*}(\dd)$.

Assume first of all that $\Gamma_i$ is a bridge. Then $\Gamma_i$ is an edge of $\Gamma$, of some type $(0,P)$ or $(h,P)$. Let $C_1,C_2$ be the vertices of $\Gamma_i$. Assume that the type is $(0,P)$, and that $C_1$ has genus zero. Then $\dd_i$ equals the divisor called $\dd_0^P(d)$ above. We find that the equality $g_{\Gamma_i}(\dd_i,\dd_i)=a(P,d)$ holds. Assume that the type is $(h,P)$, and that $C_1$ has genus $h$, then $\dd_i$ equals the divisor called $\dd_h^P(d)$ above. Hence we have $g_{\Gamma_i}(\dd_i,\dd_i)=a(P,h,d)$.

Next assume that $\Gamma_i$ is a $2$-connected component of $\Gamma$. By Corollary \ref{positive} we have that $g_{\Gamma_i}(\dd_i,\dd_i) \geq 0$.

Recall that the number of bridges in $\Gamma$ of type $(0,P)$ is $\ell_0^P$, and the number of bridges in $\Gamma$ of type $(h,P)$ is $\ell_h^P$. Using Proposition \ref{additivity} we obtain
\[ \begin{split}
a = g_\Gamma(\dd,\dd) & =  \sum_i g_{\Gamma_i}(\dd_i,\dd_i) \\
 & \geq \sum_{P \subseteq I, |P| \geq 2} a(P,d)\ell_0^P + \frac{1}{2} \sum_{h=1}^{g-1} \sum_{P \subseteq I} a(P,h,d)\ell_h^P
\end{split} \]
as required.
\end{proof}
\begin{remark}  The proof of Theorem \ref{hainconj} above yields the additional information that there is no height jumping for morphisms $\overline{S} \to \overline{\mm}_{g,n}$ such that the image is contained in $\overline{\mm}_{g,n} \setminus \Delta_0^{\mathrm{sing}}$. For morphisms $\overline{S} \to \overline{\mm}_{g,n}$ with image contained in $\overline{\mm}_{g,n} \setminus \Delta_0$ this can alternatively be deduced from the fact that the variation of polarised Hodge structure that gives rise to $\jj_{g,n}$ and the normal function sections $F_{d}$ extend naturally over $\overline{\mm}_{g,n} \setminus \Delta_0$, the moduli space of stable curves of compact type, cf. \cite[Section~5]{hainnormal}.
\end{remark}

\section{Explicit formula} \label{section:calc}

It is straightforward to compute the numbers $a(P,d)$ and $a(P,h,d)$ and thus to write down the Lear extension of $\mathcal{L}$ from $\mm_{g,n}$ over $\overline{\mm}_{g,n}$ explicitly. Applying Proposition \ref{green_r} we find
\[ a(P,d) = (m+\sum_{i \in P}d_i)^2 \, , \quad a(P,h,d) = (-(2h-1)m+\sum_{i \in P}d_i)^2 \, . \]
The formula in Theorem \ref{lear_moduli_space} becomes
\[ \mathcal{L}_{\overline{Y}} = -\pair{D,D}-\sum_{P \subseteq I, |P| \geq 2} (m+\sum_{i \in P}d_i)^2\delta_0^P - \frac{1}{2}\sum_{h=1}^{g-1} \sum_{P \subseteq I} (-(2h-1)m+\sum_{i \in P}d_i)^2\delta_h^P  \, . \]
We can make the Deligne pairing $\pair{D,D}$ explicit as follows. We have
\[ \pair{D,D} = \pair{ \sum_i d_i[x_i] - mR, \sum_i d_i [x_i] - mR} \, , \]
which gives
\[ \pair{D,D} = \sum_{i,j} d_id_j \pair{[x_i],[x_j]} - 2m \sum_i d_i \pair{[x_i],R} + m^2 \pair{R,R} \, . \]
Noting that $\pair{[x_i],[x_j]}$ is trivial if $i \neq j$, and $\pair{[x_i],[x_i]}=-\pair{[x_i],R}$ by the adjunction formula, we obtain
\[ \pair{D,D} = -\sum_i (d_i^2 + 2md_i)\pair{[x_i],R} + m^2\pair{R,R} \, . \]
This leads to
\[ \begin{split}
\mathcal{L}_{\overline{Y}} = & -m^2\pair{R,R} + \sum_i (d_i^2 + 2md_i)\pair{[x_i],R} \\
& -\sum_{P \subseteq I, |P| \geq 2} (m+\sum_{i \in P}d_i)^2\delta_0^P - \frac{1}{2}\sum_{h=1}^{g-1} \sum_{P \subseteq I} (-(2h-1)m+\sum_{i \in P}d_i)^2\delta_h^P  \, . \end{split} \]
Let $\kappa_1$ be the pullback to $\overline{Y}$ of the $\kappa$-class on $\overline{\mm}_g$. For each $i \in I$ denote by $p_i \colon \overline{Y} \to \overline{\mm}_{g,1}$ the projection keeping the $i$-th marked point, and let $\psi_i$ denote the pullback along $p_i$ of the $\psi$-class on $\overline{\mm}_{g,1}$. By the Noether formula and the remarks on \cite[p. 161]{ac} these classes are related to our earlier $\pair{R,R}$ and $\pair{[x_i],R}$ via the relations
\[ \pair{R,R} = \kappa_1 - \sum_{|P| \geq 2} \delta_0^P \, , \quad \pair{[x_i],R}=\psi_i + \sum_{P \ni i, |P| \geq 2} \delta_0^P \, . \]
We end up with the identity
\[ \begin{split}
\mathcal{L}_{\overline{Y}} = & -m^2 \kappa_1 + \sum_i (d_i^2+2md_i)\psi_i \\
& - 2 \sum_{P \subseteq I} \sum_{ \{ x_j, x_k \} \subset P, j \neq k} d_jd_k \delta_0^P - \frac{1}{2} \sum_{h=1}^{g-1} \sum_{P \subseteq I} (-(2h-1)m+\sum_{i \in P}d_i)^2\delta_h^P
\end{split} \]
in $\mathrm{Pic}(\overline{\mm}_{g,n})$. This gives an alternative proof of \cite[Theorem 11.5]{hainnormal}.

\vspace{0.5cm}

\noindent Address of the authors:\\ \\
Mathematical Institute  \\
University of Leiden  \\
PO Box 9512  \\
2300 RA Leiden  \\
The Netherlands  \\ \\
Email: \verb+{holmesdst,rdejong}@math.leidenuniv.nl+


\begin{thebibliography}{99}

\bibitem{ar} S. Y. Arakelov, \emph{An intersection theory for divisors on an
arithmetic surface}. Izv. Akad. USSR 86 (1974), 1164--1180.

\bibitem{acg} E. Arbarello, M. Cornalba, P. Griffiths, \emph{Geometry of Algebraic Curves. Volume II}. Springer Verlag, 2011.

\bibitem{ac} E. Arbarello, M. Cornalba, \emph{The Picard groups of the moduli spaces of curves}. Topology 26 (1987), 153--171.

\bibitem{beil} A. Beilinson, \emph{Height pairing between algebraic cycles}. In: Current trends in arithmetical algebraic geometry (Arcata, Calif., 1985), Contemp. Math. 67 (1987), 1--24.

\bibitem{bl} C. Birkenhake, H. Lange, \emph{Complex abelian varieties}. Grundlehren der mathematischen Wissenschaften vol. 302. Springer Verlag 2004.

\bibitem{blr} S. Bosch, W. L\"utkebohmert, M. Raynaud, \emph{N\'eron models}. Ergebnisse der Mathematik und ihrer Grenzgebiete 3. Folge, Band 21. Springer-Verlag 1990.

\bibitem{bost} J.-B. Bost, \emph{Green's currents and height pairing on complex tori}.  Duke Math. J. 61 (1990), no. 3, 899--912.

\bibitem{brospearl} P. Brosnan, G. Pearlstein, \emph{Jumps in the archimedean height}. Unpublished manuscript, 2006.

\bibitem{chai} C.-L. Chai, \emph{Siegel moduli schemes and their compactifications over $\cc$}. In: G. Cornell, J. Silverman (eds.), Arithmetic Geometry, Springer Verlag, 1986.

\bibitem{de} P. Deligne, \emph{Le d\'eterminant de la cohomologie}. In: Current trends in arithmetical algebraic geometry (Arcata, Calif., 1985), Contemp. Math. 67 (1987), 93--177.

\bibitem{fc} G. Faltings, C. Chai, \emph{Degeneration of Abelian Varieties}. Ergebnisse der Mathematik und ihrer Grenzgebiete, 3. Folge, vol. 22, Springer Verlag, 1990.

\bibitem{gs} H. Gillet, C. Soul\'e, \emph{Intersection sur les vari\'et\'es d'Arakelov}. C.R. Acad. Sci. Paris S\'er I Math. 299 (1984), 563--566.

\bibitem{gr} B. Gross, \emph{Local heights on curves}. In: G. Cornell, J. Silverman (eds.), Arithmetic Geometry, Springer Verlag, 1986.

\bibitem{SGA7} A. Grothendieck, \emph{S\'eminaire de G\'eom\'etrie Alg\'ebrique du Bois Marie - 1967-69 - Groupes de monodromie en g\'eom\'etrie alg\'ebrique - (SGA 7) - vol. 1}. Lecture Notes in Mathematics 288 (1972). Berlin; New York: Springer-Verlag.

\bibitem{hainbiext} R. Hain, \emph{Biextensions and heights associated to curves of odd genus}. Duke Math. J. 61 (1990), 859--898.

\bibitem{torelli} R. Hain, \emph{Torelli groups and geometry of moduli spaces of curves}.  In: Current topics in complex algebraic geometry (Berkeley, CA, 1992/93), 97--143, Math. Sci. Res. Inst. Publ., 28, Cambridge Univ. Press, Cambridge, 1995.

\bibitem{hrar} R. Hain, D. Reed, \emph{On the Arakelov geometry of moduli spaces of curves}. J. Differential Geom.  67 (2004), 195--228.

\bibitem{hainnormal} R. Hain, \emph{Normal functions and the geometry of moduli spaces of curves}. In: G. Farkas and I. Morrison (eds.), Handbook of Moduli, Volume I.  Advanced Lectures in Mathematics, Volume XXIV, International Press, Boston, 2013.

\bibitem{lear} D. Lear, \emph{Extensions of normal functions and asymptotics of the height pairing}. PhD thesis, University of Washington, 1990.

\bibitem{liu} Q. Liu, \emph{Algebraic Geometry and Arithmetic Curves}. Oxford University Press, 2006.

\bibitem{mb} L. Moret-Bailly, \emph{M\'etriques permises}. In: S\'eminaire sur les pinceaux arithm\'etiques, Ast\'erisque 127 (1985), 29--88.

\bibitem{nak} I. Nakamura, \emph{Relative compactification of the N\'eron model and its application}. In: W. Baily, T. Shioda (eds.), Complex Analysis and Algebraic Geometry. Cambridge University Press 1977.

\bibitem{pearl} G. Pearlstein, \emph{$\mathrm{SL}_2$-orbits and degenerations of mixed Hodge structure}. J. Diff. Geometry 74 (2006), 1--67.

\bibitem{si} C. L. Siegel, \emph{Symplectic geometry}. Johns Hopkins University Press 1964.

\bibitem{vo} C. Voisin, \emph{Hodge Theory and Complex Algebraic Geometry, I}. Cambridge Studies in Advanced Mathematics, 2002.

\bibitem{we} R. Wentworth, \emph{The asymptotics of the Arakelov-Green's function and Faltings' delta invariant}.  Comm. Math. Phys.  137  (1991), 427--459.

\bibitem{zh} S. Zhang, \emph{Admissible pairing on a curve}. Invent. Math. 112 (1993), 171--193.



\end{thebibliography}
\end{document}